\documentclass{amsart}
\usepackage[utf8]{inputenc}
\usepackage[T1]{fontenc}
\usepackage[english]{babel}
\usepackage{color}
\usepackage{amsmath} 
\usepackage{amssymb} 
\usepackage{amsthm}
\usepackage{graphicx}
\usepackage[colorlinks=true, linkcolor=blue]{hyperref}
\usepackage{cancel}
\usepackage{lmodern}
\usepackage{microtype}
\usepackage{graphicx}

\theoremstyle{plain}
\newtheorem{thm}{Theorem}[section]
\newtheorem*{thm*}{Theorem}
\newtheorem{prop}[thm]{Proposition} 
\newtheorem{lem}[thm]{Lemma} 
\newtheorem{cor}[thm]{Corollary}
\newtheorem*{cor*}{Corollary}

\newtheorem{defi}[thm]{Definition}

\newcommand {\R} {\mathbb{R}} \newcommand {\Z} {\mathbb{Z}}
\newcommand {\T} {\mathbb{T}} \newcommand {\N} {\mathbb{N}}

\newcommand {\p} {\partial}
\newcommand {\dt} {\partial_t}
\newcommand {\sgn} {\text{sgn}}

\begin{document}
\title[Echo Chains as a Linear Mechanism]{Echo Chains as a Linear Mechanism:
  Norm Inflation, Modified Exponents and Asymptotics}
\author{Yu Deng}
\address{Department of Mathematics, University of Southern California, Los
  Angeles, CA 90089, USA}
\email{yudeng@usc.edu}
\author{Christian Zillinger}
\address{BCAM -- Basque  Center  for  Applied  Mathematics, Mazarredo 14, E48009
  Bilbao, Basque Country -- Spain}
\email{czillinger@bcamath.org}

\begin{abstract}
  In this article we show that the Euler equations, when linearized around a low
  frequency perturbation to Couette flow, exhibit norm inflation in Gevrey-type
  spaces as time tends to infinity. Thus, echo chains are shown to be a (secondary)
  linear instability mechanism. 
  Furthermore, we develop a more precise analysis of cancellations in the
  resonance mechanism, which yields a modified exponent in the high frequency
  regime. In addition it allows us to remove a logarithmic constraint on the
  perturbations present in prior works by Bedrossian, Deng and Masmoudi and to
  construct solutions which are initially in a Gevrey class for which the
  velocity asymptotically converges in Sobolev regularity but
  diverges in Gevrey regularity.
\end{abstract}

\maketitle

\tableofcontents
\section{Introduction}
\label{sec:intro}

In this article our aim is to develop a further understanding of the long-time
asymptotic behavior of the 2D Euler equations near Couette flow
\begin{align}
  \label{eq:Couette}
  \p_t \omega + y \p_x \omega + v \cdot \nabla \omega =0
\end{align}
in an infinite channel $\T \times \R$.
As the velocity field is incompressible all $L^p$ norms of the vorticity
$\omega$ are conserved and the equation is nonlinearly stable in this sense.
Furthermore, the linearized problem exhibits weak convergence of the vorticity
and as consequence in the linear problem $v - \langle v \rangle_x$ converges strongly as $t
\rightarrow \infty$.
This behavior is known as linear inviscid damping in analogy to the similar
phenomenon of Landau damping in plasma physics.
While the linearized problem possesses an explicit solution and exhibits linear
inviscid damping for any initial data in $H^s, s\geq 0$, the question of
stability and asymptotic behavior of more general shear flows or the nonlinear
problem have been a very active area of research in recent years. In particular,
we mention the following publications:
\begin{itemize}
\item In \cite{bedrossian2015inviscid} Bedrossian and Masmoudi established
  nonlinear inviscid damping for Gevrey $2$ regular perturbations around Couette
  flow in an infinite periodic channel. Their method of proof has further been
  extended to the setting of Landau damping \cite{bedrossian2013landau},
  \cite{Villani_long}.
\item These results were recently extended to the case of compactly
  supported Gevrey regular perturbations to Taylor-Couette flow in
  \cite{ionescu2018inviscid}.
\item In \cite{deng2018} the first author and Masmoudi constructed solutions of
  the Euler equations, s.t. the initial data is Gevrey $2$ close to Couette
  flow, the solution is well-understood for a finite time, within that finite
  time the solution exhibits growth along echo chains consistent with a loss of
  Gevrey $2$ regularity.
  Here, we in particular stress that due to the challenging control of nonlinear
  corrections, these solutions do not capture the full echo chains (the finite
  time includes about half of a chain) and hence do not rule out subsequent
  asymptotic stability. Furthermore, this work imposes a logarithmic smallness
  restriction (compare our results in Section \ref{sec:Model}). Removing this
  restriction is a key challenge in establishing our modified scattering in
  Section \ref{sec:mod}.
\item Concerning the problem of linear inviscid damping, the case of a
  finite-periodic channel was shown to behave qualitatively differently in a
  work by the second author \cite{Zill5} in that stability results are limited
  to (sharp) $H^{3/2-}$ or weighted $H^2$,   \cite{Zhang2015inviscid},
  Sobolev regularity of the vorticity. While sufficient to establish linear inviscid
  damping with the optimal decay rates of the velocity perturbation, this leaves
  a large gap compared to the Gevrey regularity requirement of existing
  nonlinear result.
  Indeed, in \cite{jia2019linear} Jia further imposes a compact support
  assumption to establish linear stability in Gevrey spaces. 
\item This gap together with the physical observation of damping suggests that
  damping of the velocity field might be a more robust mechanism than the
  stability of the vorticity. A first modest step in this direction can be found
  in \cite{zillinger2018forced} where the second author considered special types
  of forcing of the 2D Euler and Navier-Stokes equations and showed that
  damping may persist despite instability of the vorticity.
\end{itemize}
In this article, we stress the point of view that for the question of nonlinear
inviscid damping in addition to the question of linear stability of shear flows
in Sobolev and Gevrey regularity, one should consider the following two
observations:
\begin{enumerate}
\item Existing works study the infinite time asymptotic stability of the
  vorticity in higher Sobolev regularity or Gevrey regularity, from which
  (linear) inviscid damping then follows as a corollary.
  However, this is a strictly stronger condition than the physically observed
  phenomenon of inviscid damping, which is the convergence of the velocity
  field. Indeed, we show that the linearized
  problem exhibits solutions which exhibit norm inflation due to complete echo
  chains. Here, the velocity asymptotically converges despite the divergence of the vorticity as time tends to infinity.
\item In the study of nonlinear inviscid damping, the main challenge is given by
  cascades of resonances, also known as echo chains. This mechanism is, however,
  not present in the linearized problem around shear flows due to the decoupling
  structure in Fourier space of these equations.
  In this article, we identify echo chains as a secondary linear mechanism,
  where we linearize around an arbitrarily close low-frequency perturbation
  around Couette flow. For simplicity of calculations and presentation, in this
  article we fix this low-frequency perturbation as a single mode perturbation
  by $c\cos(x)$ with $c$ small. However, we think that an extension to more
  general low-frequency perturbations, while technically tedious, should follow
  by similar arguments.
  We show that this linearized problem exhibits the same echo chains and
  norm-inflation results as for the nonlinear problem, globally in time.
  Thus, \emph{echo chains are a linear mechanism}.
  Furthermore, we identify a critical regularity threshold at which stability of
  the vorticity in Gevrey regularity fails, but where damping of the velocity
  nevertheless persists.
\end{enumerate}
In order to introduce our model consider the Euler equations near Couette flow
\eqref{eq:Couette} and change to Lagrangian coordinates (with respect to Couette
flow) $(x+ty,y)$.
Then our equation is given by
\begin{align*}
  \dt \omega + \nabla^\perp \Delta_{t}^{-1} \omega \cdot \nabla \omega=0,
\end{align*}
where $\Delta_t=\p_x^2+(\p_y-t\p_x)^2$ and $\omega$ denotes the perturbation.
Let now $P_N$ denote a Littlewood-Payley projection to a
large dyadic scale $N$.
Then the projection of the equation yields
\begin{align}
  \dt P_N \omega + P_N (\nabla^\perp \Delta_t^{-1}\omega \cdot \nabla \omega)=0.
\end{align}
We may then further decompose the nonlinearity by frequency-localizing each
factor:
\begin{align}
  \sum_{N_1+N_2 \approx N} P_N  (\nabla^\perp \Delta_t^{-1}P_{N_1}\omega \cdot \nabla P_{N_2}\omega)
\end{align}
Here, we obtain several regimes. If $N_1 \approx N_2$, we may freely integrate
by parts and trade the inverse Laplacian $\Delta_t^{-1}$ for time decay.
Similarly, if $N_2\approx N, N_1\ll N_2$, we may use the incompressibility and
the inverse Laplacian to obtain a time-decreasing energy estimate.

The main source of possible growth and instability in the nonlinear problem is
thus given by the interaction of the \emph{low-frequency part of the vorticity}, $N_2
\ll N$, and the \emph{high-frequency part of the velocity}, $N_1 \approx N$.
In our model we now \emph{fix the low frequency part} $\omega_{\text{low}}=c\cos(x)$ with $c$ small and
consider
\begin{align}
  \label{eq:linearmodel}
  \dt \omega + c \sin(x) \p_y \Delta_{t}^{-1} \omega=0.
\end{align}
We stress that $\omega_{\text{low}}$ is a stationary solution (at the level of
the vorticity) of the Euler equations and that our model thus corresponds to a
secondary linearization around $v=(y,0)+ c
\nabla^{\perp}\frac{1}{1+t^2}\cos(x-ty)$, where we neglected the transport by
$\frac{c\sin(x)}{1+t^2}\p_y$ for simplicity. 

Our main result is then that this model exhibits the same \emph{echo chains} as
the full nonlinear Euler equations near Couette flow (in Gevrey regularity, see
\cite{deng2018}, \cite{bedrossian2013asymptotic}), but further exhibits
\emph{modified scattering} and linear \emph{inviscid damping} in the sense
the velocity perturbation strongly converges as $t\rightarrow \infty$, but that
the vorticity diverges in $H^{s}$ for any $s>-1$.

\begin{thm}[Summary]
  Let $0<c<0.2$, then there exists $C>0$ such that for any $s\in \R$ there exists
  solutions of \eqref{eq:linearmodel} with $\omega_0 \in
  \mathcal{G}_{C,\frac{1}{2}}$ (Gevrey class, see Section~\ref{sec:notation}) such that $\omega(t)$ converges to a limit
  $\omega_{\infty}$ in $H^{\sigma}, \sigma\leq s$, but diverges in $H^{\sigma},
  \sigma>s$. In particular, choosing $s\geq 0$ damping, that is asymptotic
  convergence of the velocity field, holds while asymptotic stability in Gevrey
  regularity fails.  
\end{thm}

Our article is structured as follows:
\begin{itemize}
\item In Section \ref{sec:LWP} we introduce some changes of coordinates and
  establish local and global well-posedness of the model (with highly sub-optimal
  exponential growth bounds).
\item In Section \ref{sec:chains} we introduce the echo chain mechanism for what
  we call a ``homogeneous'' model, which neglects higher-order nearest neighbor
  interactions. This model is similar to the toy model of
  \cite{bedrossian2013inviscid} but additional takes into account cancellations
  between different modes.
  Here, the evolution turns out to be connected to estimates on the 1D
  Schrödinger problem with scaling-critical potentials.
\item In Section \ref{sec:Model} we show that for $c$ small compared to $\eta$,
  which is similar to the conditions imposed in \cite{deng2018},
  \cite{bedrossian2015inviscid}, also the full model exhibits echo chains and
  norm inflation. However, since this condition limits considerations to finite
  $\eta$ (where all norms are equivalent) a key challenge and improvement of the
  following section is to remove this restriction.
\item In Section \ref{sec:echo} we then show that the full problem satisfies the same growth.
  In particular, we show that for large $\eta$ (compared to $c^{-1}$) the
  dependence on $\eta$ differs from the one in \cite{deng2018}, \cite{bedrossian2015inviscid}  (it is given by a
  modified power law) and construct solutions which are global in time and
  exhibit inviscid damping (that is convergence of the velocity) but whose
  vorticity asymptotically diverges.
\end{itemize}
We remark that in the high regularity, small $c,\eta$ regime the nonlinear
problem and our model problem both asymptotically approach transport dynamics.
It remains a challenging problem to determine how the full
nonlinear Euler equations near Couette flow behave outside this regime and
whether they also exhibit modified asymptotics similar to this linear model. 

\subsubsection{Notation}
\label{sec:notation}
We use the following notational conventions:
\begin{itemize} 
\item The spatial Fourier transform of a function $\omega(t, x, y) \in
  L^2 (\T\times \R)$ is denoted
  by $\tilde{\omega}(t,k,\eta) \in L^2(\Z \times \R)$.
\item 
The Gevrey class $\mathcal{G}_{C,\frac{1}{s}}$ is defined in terms of the Fourier transform, that is $u \in
\mathcal{G}_{C,\frac{1}{s}}$ if $\exp(C|\eta|^{s})\tilde{u} \in L^2$. Here,
we omit the $k$ dependence in the exponent since for the functions we consider
only the region $|k|\leq |\eta|$ is of interest.
\item We use $a \lesssim b$ to denote that there exists an absolute constant
  $C>0$ such that $|a|\leq C |b|$. In particular, we omit absolute value signs
  in our notation.  
\item In our calculations $c \in (0,0.2) $ and $\eta 
\in \R $ can be treated as arbitrary but fixed
  parameters.
  While the low frequency regime allows for rather simple arguments (see
  Section \ref{sec:chains}), in the high frequency regime where $|\eta|$ is much
  larger than $c^{-1}$ much finer control of cancellations is necessary.
\item There we use $a \approx b$ to denote that there exists a constant $C_1>1$,
  such that $C_1^{-1} a\leq b \leq C_1 a$ such that $|C_{1}-1|\leq 0.01$ for
  the regime of $c$ and $\eta$ we are considering (with a possibly even smaller
  deviation for $c$ smaller and/or $\eta$ larger).
  For example, in the regime of large $\eta$ it holds that $\eta^{2-c}+\eta \approx \eta^{2-c}$.
\end{itemize}

\subsubsection*{Acknowledgments}
Yu Deng acknowledges support by NSF grant DMS-1900251.

Christian Zillinger would like to thank the Max-Planck Institute for Mathematics
in the Sciences where part of this work was written for its hospitality.

Christian Zillinger's research is supported by the ERCEA under the grant 014
669689-HADE and also by the Basque Government through the BERC 2014-2017
program and by Spanish Ministry of Economy and Competitiveness MINECO: BCAM Severo Ochoa excellence accreditation SEV-2013-0323.
\section{Local Wellposedness and Asymptotic Stability}
\label{sec:LWP}
Since our equation \eqref{eq:linearmodel} involves multiplication by a sine,
it has an explicit characterization in Fourier variables:
\begin{align*}
  \dt \tilde{\omega}(t,k,\eta)&+ \frac{c\eta}{(k-1)^2+(\eta-(k-1)t)^2} \tilde{\omega}(t,k-1,\eta)\\
&- \frac{c\eta}{(k+1)^2+(\eta-(k+1)t)^2} \tilde{\omega}(t,k+1,\eta)=0.
\end{align*}
In particular, we note that this equation decouples with respect to $\eta$
(while linearizations around shear flows instead decouple with respect to $k$).
Hence, we may consider $\eta \in \R$ to be a given parameter and introduce a new
time variable $\tau$:
\begin{align}
  \label{eq:deftau}
    t = \tau \eta.
\end{align}
With respect to this variable \eqref{eq:linearmodel} can be equivalently expressed as
\begin{align}
  \label{eq:tau}
  \begin{split}
  \p_\tau \omega (\tau,k,\eta) &+ c \frac{1}{(k-1)^2} \frac{1}{\eta^{-2} +(\frac{1}{k-1}-\tau)^2} \omega(\tau,k-1,\eta) \\ &- c \frac{1}{(k+1)^2} \frac{1}{\eta^{-2} +(\frac{1}{k+1}-\tau)^2} \omega(\tau,k+1,\eta)=0.
  \end{split}
\end{align}
In particular, we note that in this formulation all \emph{critical times} are given by
$\frac{1}{k}$ and thus independent of $\eta$.
Furthermore, there are no critical times after time $\tau=1$, which implies that
norm inflation or instability results are restricted to the evolution on small
times and that the evolution is asymptotically stable after time $\tau=2$.

\begin{prop}
  \label{prop:largetime}
  Let $X$ be a weighted $L^2$ (or $l^2$) space on the Fourier side (e.g. a
  fractional Sobolev or Gevrey space) such that the weight $\rho(k,\eta)$ satisfies
  $\sup_{\eta} \frac{\rho(k\pm 1,\eta)}{\rho(k,\eta)}\leq C_1 < \infty$.
  Then there exists a constant $C=C(cC_1)$ such that any solution of
  \eqref{eq:tau} satisfies
  \begin{align*}
    \|\omega(\tau)\|_{X}\leq C \|\omega(2)\|_{X}
  \end{align*}
  for all $\tau\geq 2$.
  Furthermore, $\omega(\tau)$ strongly converges in $X$  to a limit
  $\omega_{\infty} \in X$ as $t \rightarrow \infty$.
\end{prop}

\begin{proof} [Proof of Proposition \ref{prop:largetime}]
  We may estimate the multipliers by
  \begin{align*}
    \frac{c}{l^2} \frac{1}{\eta^{-2} + (\frac{1}{l}-\tau)^2} \\
    \leq c \frac{1}{0+(\frac{1}{l}-2+(\tau-2))^2} \\
\leq \frac{c}{(1+(\tau-2))^2},
  \end{align*}
  irrespective of the size of $\eta$ or of $l \in \Z, l \neq 0$.
  We further remark that if $l \eta<0$ no restriction is necessary, since then
  there can't be any cancellation.
  We may estimate 
  \begin{align*}
    \dt \|\omega\|_{X}&\leq \frac{c}{1+(\tau-2)^2}(\|\omega(k+1)\|_{X}+ \|\omega(k-1)\|_{X}) \\
    &\leq \frac{2cC_1}{1+(\tau-2)^2} \|\omega\|_{X}. 
  \end{align*}
  The result then immediately follows from integrating this inequality with
  \begin{align*}
    C:= \exp (2cC_1 \|\frac{1}{1+(\tau-2)^2}\|_{L^1}).
  \end{align*}
\end{proof}
We remark that the admissible class of spaces $X$ includes fractional Sobolev
spaces $H^s$ and Gevrey spaces $\mathcal{G}_{C,s}$, but not the homogeneous
fractional Sobolev spaces due to the quotient $\frac{w(\eta\pm
  1)}{\omega(\eta)}$ degenerating for $\eta\downarrow 0$.

In order to analyze possible norm inflation, in the following we hence focus on
characterizing the evolution on the time interval $[0,2]$.
Here, as a first easy, non-optimal estimate we may roughly control the
multiplier by
\begin{align*}
\sup\limits_{l, l \neq 0} \frac{|c\eta^2|}{l^2}= c\eta^2
\end{align*}
 and thus obtain the following theorem.
\begin{thm}
  \label{thm:LWP}
  Let $X$ be as in Proposition \ref{prop:largetime}.
  Then if the initial data is localized on the mode $\eta$, it holds that
  \begin{align*}
    \|\omega(\tau)\|_{X} \leq \exp(2C_1 c\eta^{2} \min (2,\tau)) \|\omega_0\|_{X},
  \end{align*}
  for all $\tau\geq 0$. Furthermore, $\omega(\tau)$ strongly converges in $X$ as $\tau \rightarrow
  \infty$.
\end{thm}

\begin{proof}[Proof of Theorem \ref{thm:LWP}]
  By Proposition \ref{prop:largetime} it suffices to establish a bound for $\tau
  \leq 2$.
  Here, we estimate the Fourier multiplier in \eqref{eq:tau} by $c \eta^2$ and
  thus obtain that, for $\omega_0$ localized at frequency $\eta$, it holds that
  \begin{align*}
    \| \p_\tau \omega\|_{X} \leq c \eta^2 (\|S_{+1} \omega\|_{X}+ \|S_{-1} \omega\|_{X}) \\
    \leq 2 C_1 c \eta^2 \|\omega\|_{X},
  \end{align*}
  where $S_{\pm 1}$ denotes a shift in the $k$ dependence.
  The result then follows by Gronwall's lemma.
\end{proof}

We stress that these estimates are far from optimal. However, they allow us to
control the evolution until a small positive time $\tau_0>0$, after which a more
detailed study of resonance chains establishes more sharp bounds.

Indeed, in Section \ref{sec:echo}, we show that this linearized
model attains the same Gevrey $2$ class norm inflation results as the nonlinear
problem around Couette flow considered in \cite{deng2018},
\cite{bedrossian2015inviscid}. Furthermore, the solutions exist globally in time
and exhibit modified scattering and damping.

In order to introduce ideas, in the following Section \ref{sec:chains} we
consider the setting where $c$ is very small compared to $\eta^{-1}$.
This setting allows for highly simplified proofs and serves to introduce the
resonance mechanism in a clear way.
In Section \ref{sec:Model} we then introduce an improved ``toy model'', which includes
further cancellation mechanisms compared to the one studied in
\cite{bedrossian2013inviscid} and shows what modifications to asymptotic
convergence should be expected for large $\eta$ (compared to $c^{-1}$).
The results of Section \ref{sec:echo} then show that the full linear problem
exhibits the same kind of growth as this model and we further discuss implications
for scattering and damping.

\section{Echoes, Paths and Norm Inflation}
\label{sec:chains}
Theorem \ref{thm:LWP} of the preceding section proves wellposedness of the
evolution equation for frequency localized initial data, but provides a very rough upper bound on the possible growth.
In this section we study the evolution of equation \eqref{eq:tau}
\begin{align}
  \begin{split}
  \p_\tau \omega (\tau,k,\eta) &+ c \frac{1}{(k-1)^2} \frac{1}{\eta^{-2} +(\tau- \frac{1}{k-1})^2} \omega(\tau,k-1,\eta) \\ &- c \frac{1}{(k+1)^2} \frac{1}{\eta^{-2} +(\tau -\frac{1}{k+1})^2} \omega(\tau,k+1,\eta)=0.
  \end{split}
\end{align}
for $0\leq \tau \leq 2$ in more detail in order to obtain a finer description of
the associated norm-inflation mechanism.

Considering the first Duhamel iteration, a growth mechanism is given by
\begin{align*}
 \frac{1}{(k-1)^2} \int_{-\infty}^\infty \frac{c}{\eta^{-2}+(\tau-\frac{1}{k-1})} d\tau = c\pi \eta \frac{1}{(k-1)^2}  = c \pi \frac{\eta}{(k-1)^2}.
\end{align*}
Similarly to the nonlinear equations studied in \cite{deng2018} and
\cite{bedrossian2015inviscid}, the main norm inflation mechanism of
\eqref{eq:tau} is then for mode $(k,\eta)$ to induce growth of the mode
$(k-1,\eta)$ at time $\tau \approx \frac{1}{k}$, which in turn induces growth of
the mode $(k-2,\eta)$ at time $\tau \approx \frac{1}{k-1}$. Iterating this
heuristic along a chain until time $\tau\approx 1$, this suggests a total growth
factor of
\begin{align}
  \label{eq:growthfactor}
  \frac{\eta}{k^2} \frac{\eta}{(k-1)^2} \dots \frac{\eta}{1^2}= \frac{\eta^{k}}{(k!)^2}.
\end{align}
This factor achieves its maximum with respect to $k$ for $k \approx
\sqrt{\eta}$, which yields $\exp(C \sqrt{\eta})$ and thus a norm inflation
result consistent with Gevrey $2$ regularity.

Compared to the nonlinear results of \cite{deng2018}, we highlight the following
differences of the present setting:
\begin{itemize}
\item Due to the difficulties involved in controlling a nonlinear evolution
  around growing perturbations, \cite{deng2018} establishes such a growth chain
  until time $t \approx \frac{1}{2k}\approx \frac{1}{2\sqrt{\eta}}\ll 1$. Our
  results established for the linear equation
  \eqref{eq:tau} instead capture a full echo chain as well as the subsequent
  asymptotic stability as $\tau \rightarrow \infty$.
\item The more explicit structure of the associated Duhamel iteration scheme allows
  us to establish more precise upper and lower bounds and in particular quantify
  the dependence on the parameter $c$, i.e. the size of the low-frequency part.
  Furthermore, as we discuss in Section \ref{sec:Model}, the evolution for
  $\eta$ much larger than $k^2$ introduces additional logarithmic corrections,
  which we show in Section \ref{sec:echo} to result in a correction to the
  asymptotic evolution.
  In particular, prior works included a constraint of the form
  $c\log(1+|\eta|)\lesssim 1$, where a change in exponent by $c$ is then
  comparable to multiplication by a constant.
\item Identifying the sharp growth behavior and removing this constraint allows us to construct families of
  initial data, for which the velocity field converges while the vorticity
  diverges (see Section \ref{sec:echo})! We view this as an
  indication that nonlinear inviscid damping, understood as asymptotic
  convergence of the velocity field, might still hold in lower
  regularity despite norm inflation results for the vorticity.
\end{itemize}

As a first step, in the following theorems we improve the bounds by $\exp(C \eta^2)$ to Gevrey
$2$-type estimates of the above form.
Similarly to \cite{deng2018} and \cite{bedrossian2013asymptotic} we here at
first impose
a logarithmic smallness condition $c \leq C \ln(|\eta|)^{-1}$, which greatly
simplifies this proof.
However, for any fixed $c$ this prevents us from considering (sequences of)
arbitrarily large $\eta$, which is necessary to establish instability and
modified scattering results.
As we discuss in Sections \ref{sec:Model} and \ref{sec:inhomogeneous}, this is
not only a technical limitation of the proof, but for larger $\eta$ the
asymptotics indeed differ.
In Section \ref{sec:mod} we show that taking into account these modified
asymptotics our linear system exhibits norm inflation and echo chains for all
$\eta$, as well as modified scattering.

\begin{thm}[Resonance chain]
\label{thm:resonancechain}
  Let $\eta>1$, $l \in \N$ with $\frac{\eta}{l^2}\geq 1$ and let $0<c\leq
  \min((\ln(1+\eta^2))^{-1}, \frac{1}{100})$.
  Let further $\tau_0=\frac{1}{2}(\frac{1}{l+1}+\frac{1}{l})$ and $\tau_1=1.5$.
  Then the solution $\omega(\tau)$ of \eqref{eq:tau} with
  \begin{align*}
    \omega(\tau_0,k)=\delta_{k,l}
  \end{align*}
  satisfies
  \begin{align*}
    \omega(\tau_1,k) &\geq c^{l} (1-\frac{c}{1-c})^l \frac{\eta^{l}}{(l!)^2} \text{ if } k \in \{1,3\}, \\
    \omega(\tau_1,k) &\leq  c^{l} (1+\frac{c}{1-c})^l \frac{\eta^{l}}{(l!)^2} \text{ if } k \in \{1,3\}, \\
    |\omega(\tau_1,k)| &\leq c^{|k-l|} c^{l} (1+\frac{c}{1-c})^l \frac{\eta^{l}}{(l!)^2} \text{ if }k \not \in \{1,3\} .
  \end{align*}
  Furthermore, there exists $\omega_{\infty}(k)$ such that
  $\omega(t,k)\rightarrow \omega_{\infty}(k)$ as
  $\tau \rightarrow \infty$ and $\|\omega_{\infty}\|_{x}\leq C_X
  \|\omega(\tau_1)\|_{X}$. In particular, the associated velocity field
  converges as $\tau \rightarrow \infty$.
\end{thm}

We remark that such a result also yields bounds for frequency-localized initial
data globally in time by using Theorem \ref{thm:LWP} to control the evolution
until time $\tau_0$.
In order to prove Theorem \ref{thm:resonancechain} we use precise description of the evolution
operator around single resonances $\tau\approx \frac{1}{k}$ for $k=l,l-1, \dots,
1$, which we then combine to establish the over all growth.
\subsubsection{Single resonance estimates}
\label{sec:singleresonance}
We note that in equation \eqref{eq:tau} modes only directly interact with their
nearest neighbors. Thus, if we prescribe that $\omega(\tau_0)=e^{ik_0 x +i \eta
  y}$, the value of $\tilde{\omega}(\tau_1,k,\eta)$ involves sequences
$\gamma=(k_0,k_0-1,\dots, k)$, where a mode $\gamma_i$ influences a mode
$\gamma_{i+1}$ by nearest neighbor interaction.
Indeed, any finite Duhamel iteration can be associated with a sum over all such
sequences with an upper bound on their length and integrals of the form
\begin{align}
  \label{eq:pathintegral}
  \iint_{\tau_0 \leq s_1 \leq s_2 \leq \dots \leq \tau_1} \prod_{i=1}^{|\gamma|}\frac{1}{\gamma_i^2}\frac{c}{\eta^{-2}+(s_i-\frac{1}{\gamma_i})^2}.
\end{align}
In order to simplify discussion, we introduce some notation.
\begin{defi}
  Let $k,k_0 \in \Z$ with $\sgn(k)=\sgn(k_0)$. Then a \emph{path} $\gamma$ from
  $k_0$ to $k$ is a finite sequence $\gamma=(\gamma_i)_{i=1}^n$ with
  $\gamma_1=k_0, \gamma_n=k$ and $|\gamma_{i+1}-\gamma_i|=1$.
  We call the integral \eqref{eq:pathintegral} the \emph{integral} associated to
  $\gamma$, $I[\gamma]$ and call $n-1=:|\gamma|$ the \emph{length} of $\gamma$.
\end{defi}
We remark that equation \eqref{eq:tau} introduces a post-multiplication by
$\sin(x)$ and a single Duhamel iteration hence results in paths $(k_0,k_0+1)$ and
$(k_0,k_0-1)$, which we thus denote as paths of length $1$.

Given a time interval $\frac{1}{l+1}< \tau_0< \frac{1}{l} < \tau_1 <
\frac{1}{l-1}$, we observe that integrals tend to be small unless $\gamma_i=l$
for some $i$.
Indeed, if $\gamma_i \neq l$, then $|\frac{1}{\gamma_i}-\frac{1}{l}| \geq
\frac{c}{l^2}$ and we may estimate each integral with respect to
$s_i$ by 
\begin{align}
  \label{eq:1}
  \begin{split}
  \frac{1}{l^2} \int_{|\tau-\frac{1}{l}|\geq \frac{1}{2l(l-1)}} \frac{c}{\eta^{-2}+(\tau-\frac{1}{l})^2} \\
  \leq  \frac{1}{l^2} \int_{s\geq \frac{1}{2l(l-1)}} \frac{c}{\tau^2} d\tau \leq 4 c,
  \end{split}
\end{align}
uniformly in $\eta$ and in $l$.

\begin{defi}
  Let $l \in \N$ and define $T_0:=\frac{1}{2}(\frac{1}{l+1}+\frac{1}{l})<
  \frac{1}{l}<\frac{1}{2}(\frac{1}{l-1}+\frac{1}{l-2})=:T_1$.
  Let further $k, k_0 \in \N$ and let $\gamma$ be a path from $k_0$ to $k$.
  We call $\gamma_i$ \emph{resonant} if $\gamma_i=l$ and \emph{non-resonant}
  else.
  A path $\gamma$ is called \emph{non-resonant} if $\gamma_i$ is non-resonant
  for all $i$.
\end{defi}

\begin{prop}[Single resonance bound]
  \label{prop:resonance}
  Let $l \in \N$, $\eta>0$ and let
  \begin{align*}
  \tau_0:=\frac{1}{2}(\frac{1}{l+1}+\frac{1}{l})<\frac{1}{l}< \frac{1}{2}(\frac{1}{l}+\frac{1}{l-1})=: \tau_1,
  \end{align*}
  and suppose the $c>0$ is sufficiently small that $2c \leq \ln(1+ \eta)^{-1}$.
  \begin{enumerate}
  \item Let $k_0 \in \N, k_0\neq l$ and let $\omega(\tau, \cdot)$ be the solution of \eqref{eq:tau} with
  \begin{align*}
    \omega(\tau_0,k)= \delta_{k, k_0}.
  \end{align*}
  Then it holds that
  \begin{align*}
    |\omega(\tau_1,k)- \omega(\tau_0,k)| \leq C (c^{|k-l|+|k_0-l|} \frac{\eta}{l^2} + c^{|k-k_0|})
  \end{align*}
  for all $k \in \N$.
  Here, with slight abuse of notation $|k-l|$ denotes the length of the shortest
  path connecting $k$ and $l$. This agrees with the absolute value if $k\neq
  l$, but is $2$ is $k=l$, since the shortest path $(k,k\pm 1, k)$ has length $2$.
\item Let $\omega(\tau,\cdot)$ be the solution of \eqref{eq:tau} with
    \begin{align*}
    \omega(\tau_0,k)= \delta_{k,l}.
  \end{align*}
    Then it holds that
    \begin{align*}
      \omega(\tau_1,l\pm 1) &\approx \frac{\eta}{l^2}, \\
      |\omega(\tau_1,l)-1| &\leq C c^2 \frac{\eta}{l^2}, \\
      |\omega(\tau_1,k)| & \leq C (c^{|k-l|}(1+\frac{\eta}{l^2})) \text{ else}.
  \end{align*}
  \end{enumerate}
\end{prop}

\begin{proof}[Proof of Proposition \ref{prop:resonance}]
  We argue by a Duhamel-iteration approach and summing over all paths
  starting in $k_0$ and ending in $k$.
  Here, we first note that if a path is non-resonant, we may estimate its
  contribution by $c^{|\gamma|}$. If we then estimate the number of such paths
  from above by $2^{|\gamma|}$, the contribution by all non-resonant paths is
  bounded by
  \begin{align*}
    \sum_{|\gamma|\geq \text{dist}} (2c)^{|\gamma|} = \frac{1}{1-2c} (2c)^{\text{dist}},
  \end{align*}
  where dist is the length of the shortest (non-resonant) path connecting $k$
  and $k_0$.
  Our main challenge in the following is thus going to be to control resonances
  and in particular the interaction between multiple resonances (that is,
  $\gamma_i=l$ for several indices $i$).

  \underline{Ad 1:} Let $k_0 \neq l$ be given and let $k \in \N$.
  Let further $\gamma$ be a path starting in
  $k_0$ and ending in $k$ and denote $j=|\gamma|$.
  We have already discussed the case of all non-resonant paths above, so
  suppose that $\gamma_i=l$ for some $i$ and let $(i_\kappa)_{\kappa=1}^n$ denote all such indices.
  We note that we first have to reach $l$ starting from $k_0$ and
  thus $i_1\geq |k_0-l|$. Similarly it follows that $|i_n-j|\leq |k-l|$.
  Furthermore, $i_{\kappa}\leq i_{\kappa+1}-2$, since consecutive entries in a
  path are distinct.

  In order to bound $I[\gamma]$, first consider two subsequent resonances
  indices $i_{\kappa}$ and $i_{\kappa+1}$. Keeping these indices fixed, we first
  integrate over the intermediate indices $s^1, \dots s^{l}$ and consider
  \begin{align*}
    \int_{s_{i_\kappa}\leq s^1\leq \dots \leq s^{l}\leq s_{i_{\kappa+1}}} \prod_{j=1}^{l}\frac{1}{\gamma_j^2} \frac{c}{\eta^{-2}+(s^j-\frac{1}{\gamma_j})^2} ds_j \leq (Cc)^l |s_{i_{\kappa+1}}-s_{i_\kappa}|,
  \end{align*}
  where we used that by assumption all these $\gamma_{j}\neq l$ are non-resonant and can be
  bounded by powers of $c$ and bounded one of the integrals by $c
  (s_{i_{\kappa+1}}-s_{i_\kappa})$.
  
  Repeating this argument for all $\kappa$, we only need to consider the
  integrals with respect to the resonances:
  \begin{align*}
    (cC)^{j} \iint_{T_0\leq s_{i_1}\leq s_{i_2}\leq \dots \leq s_{i_n}\leq T_1}
    \frac{1}{\eta^{-2}+(s_{i_n}-\frac{1}{l})^2}
    \prod_{\kappa=1}^{n-1}\frac{s_{i_{\kappa+1}}-s_{i_\kappa}}{\eta^{-2}+(s_{i_\kappa}-\frac{1}{l})^2}.
  \end{align*}
  Rescaling and shifting all $s_{\cdot}$ as $\eta^{-1}(s-\frac{1}{l})$, we obtain a factor $\eta^{-j}$ from the
  Jacobian, a factor $\eta^{-j+1}$ from the linear factors and a factor
  $\eta^{2j}$ from the denominators and thus in total:
  \begin{align*}
        (cC)^{j} \eta \iint_{\eta(T_0-\frac{1}{l})\leq s_{i_1}\leq s_{i_2}\leq \dots \leq s_{i_n}\leq \eta(T_1-\frac{1}{l})}
    \frac{1}{1+s_{i_n}^2}
    \prod_{\kappa=1}^{n-1}\frac{s_{i_{\kappa+1}}-s_{i_\kappa}}{1+s_{i_\kappa}^2}.
  \end{align*}
  It remains to estimate the integral over this now $\eta$ dependent region.

  Expanding the product in the numerator and looking at each summand separately
  we need to consider monomials in the numerator.
  If a factor $s_{i_\kappa}$ does not appear in the numerator we simply bound by
  \begin{align*}
    \int_{\R} \frac{1}{1+s_{i_\kappa}^2}=\pi. 
  \end{align*}
  If it appears once, we may compute
  \begin{align*}
    \int_{s} \frac{|s|}{1+s^2}=\frac{1}{2}\ln(1+s^2)+c.
  \end{align*}
  If it appears twice, we bound
  \begin{align*}
    \int_{s_{i_{\kappa-1}}\leq s_{i_\kappa}\leq s_{i_{\kappa+1}}} \frac{s_{i_{\kappa}}^2}{1+s_{i_\kappa}^2} \leq s_{i_{\kappa+1}-s_{i_{\kappa-1}}}, 
  \end{align*}
  but in this way $i_{\kappa}$ does not appear in the integral anymore and we
  can repeat the above argument with this resonance removed.
  Hence, we only need to consider monomials where each $s_{i_\kappa}$ appears
  either to power $1$ or does not appear:
  \begin{align*}
    \iint_{\eta(T_0-\frac{1}{l})\leq s_{i_1}\leq s_{i_2}\leq \dots \leq s_{i_n}\leq \eta(T_1-\frac{1}{l})} \prod_{i_{\kappa} \text{ appears}} \frac{|s_{i_{\kappa}}|}{1+s_{i_\kappa}^2} \prod_{i_{\kappa} \text{ does not}} \frac{1}{1+s_{i_\kappa}^2}.
  \end{align*}
  As discussed above the integrals over the ``does not'' case can be bounded by $\pi^{j_2}$, where $j_2=j-j_1$ denotes the number of such cases.
  The integral over the first product is bounded by a power of a logarithm:
  \begin{align}
    \frac{1}{j_1!} \left(\frac{1}{2}\ln(1+(\eta(T_1-\frac{1}{l}))^2) +\frac{1}{2}\ln(1+(\eta(T_0-\frac{1}{l}))^2)\right)^{j_1} \approx \frac{1}{j_1!} \left(C\ln(1+\frac{\eta}{l^2})\right)^{j_1}.
  \end{align}
  Recalling the additional prefactor $c^{|\gamma|}$, summing over all such $j_1$
  then leads to a bound of this contribution by 
 \begin{align*}
    (1+\frac{\eta}{l^2})^{Cc} = \exp(Cc\ln(1+\frac{\eta}{l^2}))\leq \exp(\tilde{C}),
 \end{align*}
 where we used the logarithmic smallness assumption
 \begin{align*}
   c\ln(1+\frac{\eta}{l^2})\ll 1.
 \end{align*}
 Hence, this exponential can be treated as a perturbation provided $c$ and thus
 $\tilde{C}$ is sufficiently small.
 \\
 
  \underline{Ad 2:} We proceed similarly as in case $1$, but note that
  \begin{align*}
    \int_{\tau_0}^{\tau_1} \frac{1}{l^2}\frac{c}{\eta^{-2}+(s-\frac{1}{l})} \approx c \frac{\eta}{l^2}.
  \end{align*}
  All longer paths starting in $l$ and ending in $l- 1$ have length at least $3$
  and hence the sum over all these paths can again be controlled by
  \begin{align*}
    \frac{c^2}{1-c} \frac{c \eta} {l^2}.
  \end{align*}
  We hence obtain the claimed comparability with factors $1\pm \frac{c^2}{1-c}$.
  For $k=l$, we note that the shortest non-trivial path has length $2$ and for
  $k\not \in {l-1,l,l+1}$ we argue as before except that all paths start in $\gamma_1=l$.
\end{proof}
We remark that by using the linearity of the problem this further yields a
convolution-type bound for $\omega(\tau_0)$ not being localized on a single
mode.
In the following we then combine these mode-wise upper and lower bounds as
well as the local wellposedness of Theorem \ref{thm:LWP} to construct an
explicit example of a function that exhibits the Gevrey norm inflation up to
time $\tau=1$ and afterwards is asymptotically stable and hence also exhibits
inviscid damping. In Section \ref{sec:echo} we show with considerable technical
effort that a similar but distinct result also hold if $\eta$ does not satisfy a
logarithmic bound.

\subsubsection{Proof of Theorems \ref{thm:resonancechain}}
\label{sec:proof-theorems}
Using the single-resonance results of Proposition \ref{prop:resonance} we are
now ready to prove Theorem \ref{thm:resonancechain}.

We proceed in multiple steps:
\begin{enumerate}
\item We use local well-posedness to control the evolution from time $\tau=0$ up
  to a small positive time $\tau^*=\tau^*(c,\eta)\ll 1$.
\item We next choose $l \in \N$ such that $\frac{1}{l}>\tau^*$ and $l$ is
  maximal with this property and prescribe $\omega(\tau)$ at the time
  $\tau_0=\frac{1}{2}(\frac{1}{l+1}+\frac{1}{l})<\frac{1}{l}$ to be localized a
  frequencies $(l,\eta)$.
  Proposition \ref{prop:resonance} then allows us to control the evolution up to
  time $\tau_1=\frac{1}{2}(\frac{1}{l}+\frac{1}{l-1})>\frac{1}{l}$,
  where we in particular obtain upper and lower bounds on the modes $l\pm 1$ and upper
  bounds on all other modes.
\item We then iterate this control another $l-1$ times and establish upper
  bounds on all modes and lower bounds along our chain of resonances
  $(l,l-1,l-2,\dots, 1)$ until after time $\tau=1$.
\item By our construction of the coordinate $\tau$, after this time no
  resonances appear anymore and we may use Proposition \ref{prop:largetime} to
  control the long-time asymptotic behavior.
\end{enumerate}

\begin{proof}[Proof of Theorem \ref{thm:resonancechain}]
  We note that it suffices to establish upper and lower bounds on $\omega$ until
  time $\tau=1.5$, since after that time asymptotic stability and convergence of
  the velocity field follow by Proposition \ref{prop:largetime}.
  
  Let thus $l \in \N$ be given, to be fixed later.
  By Theorem \ref{thm:LWP} we may find initial data such that at time
  $\tau_l:=\frac{1}{2}(\frac{1}{l-1}+\frac{1}{l})$ it holds that
  \begin{align}
    \label{eq:induction1}
    |\omega(\tau_l,l)| \geq 0.5 \max |\omega(\tau_l)|=: 0.5 \theta.
  \end{align}
  We then apply Proposition \ref{prop:resonance} using the linearity and a
  triangle inequality to obtain a convolution bound.
  That is, for $k \not \in \{l-1,l,l+1\}$, we obtain that
  \begin{align*}
    |\omega(\tau_{l+1},k)-\omega(\tau_{l},k)|&\leq C (c^{|k-l|}(1+\frac{\eta}{l^2}))|\tilde{\omega}(\tau_l,l)| +\sum_{k_0\neq l} |\tilde{\omega}(\tau_l,k_0)| C (c^{|k-l|+|k_0-l|} \frac{\eta}{l^2} + c^{|k-k_0|}) \\
    &\leq \theta C (c^{|k-l|}(1+\frac{\eta}{l^2}))+ \theta C c^{|k-l|}\frac{1}{1-c}\frac{\eta}{l^2} +\frac{\theta}{1-c}.
  \end{align*}
  If $k=l$, the first term is replaced by
  \begin{align*}
    Cc^2 \frac{\eta}{l^2}.
  \end{align*}
  Finally, if $k\in \{l-1,l+1\}$, we obtain that
  \begin{align*}
    |\omega(\tau_{l+1},k)\mp C_{\pm} c\frac{\eta}{l^2} \omega(\tau_l,l) - \omega(\tau_{l},k)| \leq \theta C c^{|k-l|}\frac{1}{1-c}\frac{\eta}{l^2} +\frac{\theta}{1-c},
  \end{align*}
  where $C_{\pm 1}\approx 1$.
  Recalling that $\omega(\tau_l,l)$ achieves $\theta$ within a factor $2$ and
  using that $c\frac{\eta}{l^2}\gg 1$, if follows that
  \begin{align*}
   \pm c\frac{\eta}{l^2} \omega(\tau_l,l) \approx \omega(\tau_{l+1},l \pm 1) \geq 0.5 \max |\omega(\tau_{l+1})|.
  \end{align*}
  also achieves the new maximum within a factor $2$ and is larger than the
  previous maximum by a factor $c\frac{\eta}{l^2}$.
  
  We may thus repeat our application of Proposition \ref{prop:resonance}
  iteratively until for $k=1$, we obtain that
  \begin{align*}
    \omega(\tau_1,1)\approx \frac{c^{l}\eta^{l}}{(l!)^2} \omega(\tau_l,l)
  \end{align*}
  achieves the full growth along a chain and again is comparable to the supremum
  at that time.
\end{proof}
While this result is quite useful and shows the growth mechanism, the
logarithmic bound
\begin{align*}
  c \ln(1+\frac{\eta}{l^2}) \ll 1,
\end{align*}
prevents us from considering $\eta$ arbitrarily large. In particular, reverting
the change of variables $t \mapsto \tau$, this means that after a final time
$t=\max \eta$ there are no more resonances and the evolution is asymptotically
stable. Furthermore, since we can only consider a finite set in $\eta$ all norms
are equivalent and thus the associated growth while suggestive of a Gevrey
regularity class is consistent with any norm.

We remark that also \cite{deng2018}, \cite{bedrossian2015inviscid} contain a
similar constraint. While one might at first hope that this is a purely
technical constraint, in our proof we saw various logarithmic terms appearing,
where in particular the contribution by the path $\gamma=(l,l-1,l,l-1)$ is also
bounded below. It is hence a large, non-negligible correction, which can not be
treated perturbatively.
Thus, our main goal in the following is to understand how the dynamics change
and to remove this restriction. To this end, in Section \ref{sec:Model} we first
introduce an improved model problem and new methods of proof, which allow us to
consider $\eta$ arbitrarily large. Then in Section~\ref{sec:echo} we show that
this behavior also holds in the full problem.

\section{An Improved Model Problem}
\label{sec:Model}
In order to better study the effect of resonances and, in particular, the
behavior for large $\eta$ and/or large $c$, in the following we consider an abridged model which considers the evolution for
$\frac{1}{l+1}< \tau < \frac{1}{l-1}$ and only considers the resonant mode $l$
and its neighbor $l-1$:
\begin{align*}
  \p_{\tau} \omega(\tau,l-1)- \frac{c}{\eta^{-2}+(\tau-\frac{1}{l})^2} \frac{1}{l^2} \omega(\tau,l)=0, \\
  \p_{\tau} \omega(\tau,l)+ \frac{c}{\eta^{-2}+(\tau-\frac{1}{l-1})^2} \frac{1}{(l-1)^2} \omega(\tau,l-1)=0, 
\end{align*}
In Section \ref{sec:echo} we show that also the full model behaves similar to
this two-mode model (or more accurately like the three-mode model
involving the modes $l+1,l,l-1$).

This model is very similar to the one considered in
\cite{bedrossian2013inviscid}, except that our model does not estimate the
coefficients by absolute values but rather keeps the signs, which corresponds to
exploiting the real-valuedness of $\sin(x)$ (and hence anti-symmetry of the
imaginary parts of the Fourier coefficients).

As we will see in the following this yields important cancellation
properties and further exposes a connection of this problem to the 1D Schrödinger
problem with scaling critical potential.

As a simplification of calculations, we use that $\frac{1}{l-1}$ is non-resonant
and approximate $\tau- \frac{1}{l-l}\approx
\frac{1}{l^2}$ and thus consider the following \emph{two-mode system}:
\begin{align}
  \label{eq:twomodesystem}
  \begin{split}
  \p_\tau u - \frac{c}{\eta^{-2}+\tau^2} \frac{1}{l^2} v(\tau) =0, \\
  \p_\tau v + \frac{c}{\eta^{-2}+l^{-4}} \frac{1}{l^2} u(\tau) =0,
  \end{split}
\end{align}
where we also shifted the time to $\tau \in [l^{-2}, l^2]$ for notational convenience.

In order to introduce ideas, we first consider the case where $\eta$ is not
larger than $l^2$. There, we approximate $\eta^{-2}+\tau^2 \approx
\eta^{-2}+l^{-4}\approx \eta^{-2}$ and obtain the following lemma.
\begin{lem}[Small $\eta$ case]
  Let $\eta>0, l \in \Z$ and $c \in \R$ be given and consider the following
  approximation to system \eqref{eq:twomodesystem}:
  \begin{align}
    \begin{split}
    \p_{\tau}
    \begin{pmatrix}
      u \\ v 
    \end{pmatrix}
+
    \begin{pmatrix}
      0 & -c\frac{\eta^2}{l^2} \\
      c \frac{\eta^2}{l^2} & 0
    \end{pmatrix}
                             \begin{pmatrix}
                               u \\ v
                             \end{pmatrix}
=0,
    \end{split}
  \end{align}
  for $\tau \in (0, l^{-2})$.
  Then, the unique solution is given by
  \begin{align}
    \label{eq:smalleta}
    \begin{pmatrix}
      u(\tau) \\ v(\tau)
    \end{pmatrix}
    =
    \begin{pmatrix}
      \cos(c \frac{\eta}{l^2} \tau)  & \sin(c \frac{\eta}{l^2} \tau) \\
      -\sin(c \frac{\eta}{l^2} \tau)  & \cos(c \frac{\eta}{l^2} \tau) \\
    \end{pmatrix}
    \begin{pmatrix}
      u(0) \\ v(0)
    \end{pmatrix}.
  \end{align}
 In particular, it follows that $|u(\tau)|^2+|v(\tau)|^2$ is a conserved quantity.
\end{lem}
\begin{proof}
  While the solution follows immediately by computation of the matrix
  exponential, for later reference we note an alternative proof arguing at the
  level of second derivatives.
  Formally differentiating the equation for $u$ by $\tau$ and
  using the second equation, we obtain
  \begin{align*}
    \p_\tau^2 u + c^2\frac{\eta^{4}}{l^4} u =0,
  \end{align*}
  which has a general solution
  \begin{align*}
    \alpha \cos(c \frac{\eta}{l^2} \tau)  + \beta \sin(c \frac{\eta}{l^2} \tau)
  \end{align*}
  for constants $\alpha, \beta \in \R$.
  The first-order equation
  \begin{align*}
    \p_\tau u - c \frac{\eta}{l^2} v =0,
  \end{align*}
  then further determines $\p_\tau u$ and allows us to determine $\alpha$ and
  $\beta$ in terms of $u(0), v(0)$.
\end{proof}

Having discussed the case of small $\eta$, in the following we are interested in
the setting where $\eta\geq l^2$ is potentially very large.
Here, we may again consider decoupled second order equations instead of a system
of first order equations:
\begin{align}
  \label{eq:system2}
  \begin{split}
  \p_{\tau}^2 u + \frac{c^2}{\eta^{-2}+ \tau^2} u =0, \\
  \p_\tau (\eta^{-2}+\tau^2)\p_\tau v + c^2 v =0.
  \end{split}
\end{align}
We note that the equation for $u$ is given by a stationary Schrödinger equation
with potential $\frac{c^2}{\eta^{-2}+ \tau^2}$, which is a mollified version of
the scaling critical potential $\frac{c^2}{\tau^2}$.

The equation satisfied by $v$ instead is a degenerate elliptic problem. However,
we note that \eqref{eq:twomodesystem} allows us to determine $v$ in terms of
$\p_\tau u$. Hence, in the following it suffices to study the evolution of $u$.

In Section \ref{sec:chains} we showed that a Duhamel iteration converges, with
the dominant terms being given by the nearest neighbor paths $(l,l-1,l,l-1,l,
\dots)$ of this section, but had to require that $c$ is sufficiently small to
control logarithmic corrections.
The aim in the following is to study the system \eqref{eq:twomodesystem} and
show that for large $\eta$ better estimates hold and that only an absolute bound
on $c$ is required.

We recall that \eqref{eq:system2} is posed on the interval $(-l^{-2}, l^{2})$.
We may rescale our time variable as $\tau=l^{-2} t$, so that $t \in (-1,1)$ and obtain
\begin{align}
  \label{eq:schroedinger2}
  \p_{t}^2 u + \frac{c^2}{\eta^{-2}l^4+t^2} u =0.
\end{align}
We stress that this equation depends only on $c$ and $\frac{\eta}{l^2}$, but not
on $\eta$ and $l$ separately. For simplicity of notation, in the following we
abbreviate
\begin{align}
  \label{eq:xi}
\xi:=\frac{\eta}{l^2}.
\end{align}
Once we have computed $u$, we may recover $v$ using that
\begin{align}
  \label{eq:2}
  \p_t u &= l^{-2}\p_\tau u= -c v \\
  \p_t v &= -c u 
\end{align}
by \eqref{eq:twomodesystem}. We note that here $v$ is evaluated at time
$\tau=\frac{1}{l^2}t$, but corresponds to the mode $(l-1,\eta)$ of the vorticity.
The equation \eqref{eq:schroedinger2} has an explicit solution in terms of
special functions, which we use to establish the following theorem.

\begin{thm}
  \label{thm:specialfunctions}
  Let $\xi=\frac{\eta}{l^2} \in \R$, then the problem \eqref{eq:schroedinger2}
  \begin{align*}
    \p_t^2 u + \frac{c^2}{\xi^{-2} + t^2} u =0
  \end{align*}
  on $(-1,1)$ has an explicit \emph{scattering matrix} $M=M(c,\xi)$ (see
  Proposition \ref{prop:exponentmechanism}) such
  that
  \begin{align*}
    \begin{pmatrix}
      u(1) \\ u'(1)
    \end{pmatrix}
    = M
    \begin{pmatrix}
      u(-1) \\ u'(-1) 
    \end{pmatrix}.
  \end{align*}
    In particular, if $0<c<\frac{1}{2}$ and $\xi \gg c^{-1}$, and we further
  assume that $u'(-1)\geq 0.5 |u(-1)|$, then it follows that
  \begin{align}
    \begin{pmatrix}
      u(1) \\ u'(1)
    \end{pmatrix}
    \approx \xi^{\gamma} u(-1)
    \begin{pmatrix}
      \pi c^2 \\ \pi c^2
    \end{pmatrix},
  \end{align}
  where $\gamma=\sqrt{1-8c^2}$.
\end{thm}
We thus observe that the logarithmic correction seen in Section \ref{sec:chains}
here manifests in modified exponent.
As the explicit solution in terms of hypergeometric functions is technically
involved but rather opaque, the proof of Theorem \ref{thm:specialfunctions} is
given in Appendix \ref{sec:special}.

In the following we instead discuss the underlying mechanism by approximating
$\frac{c^2}{\xi^{-2} + t^2}$ by $\frac{c^2}{t^2}$ and $c^2 \xi^2$ on
sub-intervals of $(-1,1)$. The interactions between these regimes then yields a
correction to the exponent. We remark that in this section our splitting
corresponds to a more rough approximation in order to introduce ideas.
In Section \ref{sec:setup} we choose our splitting more carefully and consider
also the influence of other modes.

Let thus $\xi \in \R$ be given and suppose that $\xi\gg c^{-1}$ is large (otherwise
we may consider the system \eqref{eq:smalleta}).
Then on the intervals $(-1,\xi^{-1})$ and $(\xi^{-1},1)$, it holds that
\begin{align*}
  \frac{1}{\xi^{-2}+t^2} \approx \frac{1}{t^2}
\end{align*}
and we hence consider the second order ODEs
\begin{align}
  \label{eq:3}
  \p_{t}^2 u + \frac{c^2}{t^2} u =0
\end{align}
on $\xi^{-1}<|t|<1$ and
\begin{align}
  \label{eq:4}
  \p_{t}^2 u + c^2 \xi^2 u =0
\end{align}
on $|t|<\xi^{-1}$.

In order to understand the mapping properties of Theorem
\ref{thm:specialfunctions} we thus consider the following scheme:
\begin{enumerate}
\item We prescribe initial data $(u(-1),u'(-1))$ and solve
  \eqref{eq:3} to obtain $(u(-\xi^{-1}),u'(-\xi^{-1}))$ in Lemma \ref{lem:schroedinger}
\item We then solve \eqref{eq:4} to obtain
  $(u(\xi^{-1}),u'(\xi^{-1}))$ in Lemma \ref{lem:innerinterval}
\item Finally, we solve \eqref{eq:3} on $(\xi^{-1},1)$ by again using Lemma \ref{lem:schroedinger}. 
\item Combining these three maps, we show in Proposition \ref{prop:exponentmechanism} that the map
  $(u,u')|_{t=-1}\mapsto (u,u')|_{t=1}$ has singular values of size
  $\xi^{\gamma}$ with $\gamma=\gamma(c)<1$.
\end{enumerate}
In the following subsection \ref{sec:2mode} we then iterate this evolution in
$l$ to obtain sharp Gevrey $2$-type norm inflation results for this model problem.

We remark that \eqref{eq:3} is the scaling critical Schrödinger problem, which has a general solution
\begin{align*}
  c_1 |t|^{\frac{1}{2}(1+\sqrt{1-4c^2})} + c_2 |t|^{\frac{1}{2}(1+\sqrt{1-4c^2})},
\end{align*}
provided $0<c < \frac{1}{2}$.
We in particular note that the exponents
$\gamma_1=\frac{1}{2}+\sqrt{\frac{1}{4}-c^2}$ and
$\gamma_2=\frac{1}{2}-\sqrt{\frac{1}{4}-c^2}$ are strictly between $0$ and $1$
and $\gamma_1+\gamma_2=1$.

\begin{lem}
  \label{lem:innerinterval}
  Let $c \neq 0, \eta >0, k \in \Z \setminus\{0\}$, then the solution
  $u \in C^2$ of
  \begin{align}
    \label{eq:5}
    \p_{t}^2 u + c^2\xi^{2} u =0
  \end{align}
  in $(-\xi^{-1}, \xi^{-1})$ satisfies
  \begin{align*}
    \begin{pmatrix}
      u(\xi^{-1}) \\ u'(\xi^{-1})
    \end{pmatrix}
=
    \begin{pmatrix}
      \cos(2c) & \frac{1}{c}\xi^{-1}\sin(2c) \\ -c \xi \sin(2c) & \cos(2c)
    \end{pmatrix}
                   \begin{pmatrix}
                     u(-\xi^{-1}) \\ u'(-\xi^{-1})
                   \end{pmatrix}.
  \end{align*}
  In particular, we observe that for $\xi \gg c^{-2}$ the bottom left matrix
  entry is by far the largest, while $\cos(2c)\approx 1$.
\end{lem}

\begin{proof}[Proof of Lemma \ref{lem:innerinterval}]
  We observe that a general solution of \eqref{eq:5} is given by $c_1 \sin(c\xi
  t) + c_2 \cos(x \xi t)$. One may then verify that
  \begin{align*}
    \begin{pmatrix}
    u(t) \\ u'(t)
    \end{pmatrix}
    =
    \begin{pmatrix}
      \cos(c\xi (t+\xi^{-1})) & \frac{1}{c}\xi^{-1}\sin(c\xi (t+\xi^{-1}))  \\-c\xi\sin(c\xi(t+\xi^{-1})) & \cos(c\xi(t+\xi^{-1}))
    \end{pmatrix}
    \begin{pmatrix}
      u(-\xi^{-1}) \\ u'(-\xi^{-1})
    \end{pmatrix}
  \end{align*}
  satisfies the initial conditions.
\end{proof}

On the exterior intervals we similarly obtain an explicit solution, but with very
different dependence on $\xi$.
\begin{lem}
  \label{lem:schroedinger}
  Let $0<c<\frac{1}{2}$ and $\xi>1$ then the solution of
   $u \in C^2$ of
  \begin{align}
    \label{eq:6}
    \p_{t}^2 u + \frac{c^2}{t^2} u =0 \text{ in }  (-1,1)\setminus (-\xi^{-1},\xi^{-1})
  \end{align}
  satisfies
  \begin{align*}
    \begin{pmatrix}
      u(\xi^{-1}) \\ u'(\xi^{-1})
    \end{pmatrix}
=
    \begin{pmatrix}
      \xi^{-\gamma_1} & \xi^{-\gamma_2} \\ \gamma_1\xi^{1-\gamma_1} & \gamma_2 \xi^{1-\gamma_2}
    \end{pmatrix}
                                                                                    \begin{pmatrix}
1 & 1 \\  \gamma_1 &  \gamma_2  
\end{pmatrix}^{-1}                               
                   \begin{pmatrix}
                     u(1) \\ u'(1)
                   \end{pmatrix},
  \end{align*}
  and
  \begin{align*}
    \begin{pmatrix}
      u(-\xi^{-1}) \\ u'(-\xi^{-1})
    \end{pmatrix}
=
    \begin{pmatrix}
      \xi^{-\gamma_1} & \xi^{-\gamma_2} \\ -\gamma_1\xi^{1-\gamma_1} & -\gamma_2 \xi^{1-\gamma_2}
    \end{pmatrix}
      \begin{pmatrix}
1 & 1 \\ - \gamma_1 & - \gamma_2  
\end{pmatrix}^{-1}                     
                   \begin{pmatrix}
                     u(-1) \\ u'(-1)
                   \end{pmatrix},
  \end{align*}
  where $\gamma_{1,2}=\frac{1}{2}\pm\sqrt{\frac{1}{4}-c^2}$.
\end{lem}
\begin{proof}[Proof of Lemma \ref{lem:schroedinger}]
  We note that a general solution of \eqref{eq:6} is given by
  \begin{align*}
    c_{1,\pm} |t|^{\gamma_1} + c_{2,\pm} |t|^{\gamma_2},
  \end{align*}
  where $c_{\cdot, \pm}$ are constant on each connected component of the domain.
  We may then again verify that
  \begin{align*}
    \begin{pmatrix}
      u(t) \\ u'(t)
    \end{pmatrix}
=
    \begin{pmatrix}
      |t|^{\gamma_1} &  |t|^{\gamma_2} \\ \sgn(t)\gamma_1 |t|^{\gamma_1-1} & \sgn(t)\gamma_2 |t|^{\gamma_2-1}
    \end{pmatrix}
\begin{pmatrix}
1 & 1 \\ \sgn(t) \gamma_1 & \sgn(t) \gamma_2  
\end{pmatrix}^{-1}
                    \begin{pmatrix}
                      u(1) \\ u'(1)
                    \end{pmatrix}
  \end{align*}
  satisfies the boundary conditions.
\end{proof}

\begin{prop}
  \label{prop:exponentmechanism}
  Let $0<c<\frac{1}{2}$ and let $u \in C^1$ be a solution of \eqref{eq:3} and \eqref{eq:4}.
  Then there exists an explicitly computable matrix $M=M(c, \xi)$ such that
  \begin{align}
    \label{eq:scattering_singleresonance}
    \begin{pmatrix}
      u(1) \\ u'(1)
    \end{pmatrix}
= M   \begin{pmatrix}
                     u(-1) \\ u'(-1)
                   \end{pmatrix}.
  \end{align}
  In particular, if $0<c<\frac{1}{2}$ and $\xi \gg c^{-1}$, and we further
  assume that $|u(-1)|\geq 2 |u'(-1)|$, then it follows that
  \begin{align}
    \begin{pmatrix}
      u(1) \\ u'(1)
    \end{pmatrix}
    \approx \xi^{\gamma} u(-1)
    \begin{pmatrix}
      8c^2 \\ 8c^2
    \end{pmatrix},
  \end{align}
  where $\gamma=\sqrt{1-8c^2}$.
\end{prop}

\begin{proof}[Proof of Proposition \ref{prop:exponentmechanism}]
  Combining Lemma \ref{lem:innerinterval} and Lemma \ref{lem:schroedinger}, we obtain that our data at $t=-1$ and $t=1$
  are related by
  \begin{align*}
    \begin{pmatrix}
      1 & 1 \\ \gamma_1 & \gamma_2
    \end{pmatrix}
                          \begin{pmatrix} \xi^{-\gamma_1} &
                            \xi^{-\gamma_2} \\ \gamma_1\xi^{1-\gamma_1} &
                            \gamma_2 \xi^{1-\gamma_2}
                          \end{pmatrix}^{-1}
                                                                                                        \begin{pmatrix}
                                                                                                          \cos(2c)
                                                                                                          &
                                                                                                          \frac{1}{c}\xi^{-1}\sin(2c)
                                                                                                          \\
                                                                                                          -c
                                                                                                          \xi\sin(2c)
                                                                                                          &
                                                                                                          \cos(2c)
                                                                                                        \end{pmatrix}\\
                                                                                                            \begin{pmatrix} \xi^{-\gamma_1} &
                                                                                                              \xi^{-\gamma_2} \\ -\gamma_1\xi^{1-\gamma_1} &
                                                                                                              -\gamma_2 \xi^{1-\gamma_2}
                                                                                                            \end{pmatrix}
                                                                                                                                                                                           \begin{pmatrix} 1 & 1 \\ - \gamma_1 & - \gamma_2
                                                                                                                                                                                           \end{pmatrix}^{-1}.
  \end{align*}
  We now compute
  \begin{align*}
    \begin{pmatrix} \xi^{-\gamma_1} &
      \xi^{-\gamma_2} \\ \gamma_1\xi^{1-\gamma_1} &
      \gamma_2 \xi^{1-\gamma_2}
    \end{pmatrix}^{-1}&= -\frac{1}{\gamma} \begin{pmatrix}
      \gamma_2 \xi^{1-\gamma_2}  & -\xi^{-\gamma_2} \\ -\gamma_1\xi^{1-\gamma_1} &
      \xi^{-\gamma_1}
    \end{pmatrix}, \\
    \begin{pmatrix}
      1 & 1 \\
      - \gamma_1 & - \gamma_2
     \end{pmatrix}^{-1} &= \frac{1}{\gamma}  \begin{pmatrix}
      -\gamma_2 & -1 \\
       \gamma_1 & 1
     \end{pmatrix},
  \end{align*}
  where we used $\gamma_1+\gamma_2=1$ and computed the determinants as
  \begin{align*}
  \gamma_1-\gamma_2=:\gamma=\sqrt{1-8c^2}\approx 1.
  \end{align*}
  Since $\xi \gg 1$, $\gamma_1=\frac{1}{2}+\sqrt{\frac{1}{4}-2c^2} \approx
  1-c^2\approx 1$, $\gamma_2= \frac{1}{2}-\sqrt{\frac{1}{4}-2c^2}\approx 2c^2\ll 1$ it follows that
  the largest powers of $\xi$ are dominant.
  We may thus approximate:
  \begin{align*}
    \begin{pmatrix}
      u(-\xi^{-1}) \\ u'(-\xi^{-1})
    \end{pmatrix}
    &=\frac{1}{\gamma} \begin{pmatrix}
      \xi^{-\gamma_1}& \xi^{-\gamma_2}
      \\ -\gamma_1\xi^{1-\gamma_1} & -\gamma_2 \xi^{1-\gamma_2}
    \end{pmatrix}
           \begin{pmatrix}
      -\gamma_2 & -1 \\
       \gamma_1 & 1
     \end{pmatrix}
                  \begin{pmatrix}
                    u(-1) \\ u'(-1)
                  \end{pmatrix}
    \\
    &\approx
    \begin{pmatrix}
      \xi^{-\gamma_1}& \xi^{-\gamma_2}
      \\ -\gamma_1\xi^{1-\gamma_1} & -\gamma_2 \xi^{1-\gamma_2}
    \end{pmatrix}
                                     \begin{pmatrix}
                                       -\gamma_2 u(-1)-u(-1) \\
                                       u(-1)+u'(-1)
                                     \end{pmatrix}\\
&\approx u(-1)
                                                       \begin{pmatrix}
                                                         \xi^{-\gamma_2} \\
                                                         -\gamma_2\xi^{1-\gamma_2}
                                                       \end{pmatrix}.
  \end{align*}
  Next, it follows that
  \begin{align*}
    \begin{pmatrix}
      u(\xi^{-1}) \\ u'(\xi^{-1})
    \end{pmatrix}
    &= \begin{pmatrix}
      \cos(2c)      &  \frac{1}{c}\xi^{-1}\sin(2c)\\
      -c \xi\sin(2c)&  \cos(2c)
    \end{pmatrix}
     \begin{pmatrix}
      u(-\xi^{-1}) \\ u'(-\xi^{-1})
    \end{pmatrix}
    \\
    &\approx u(-1)
    \begin{pmatrix}
       \xi^{-\gamma_2} \\
      -(2c^2+\gamma_2)\xi^{1-\gamma_2}
    \end{pmatrix},
  \end{align*}
  where we omitted $\gamma_2 \frac{\sin(2c)}{c}=\mathcal{O}(c^2)$ as a small
  perturbation to $1$ and approximated $\sin(2c)\approx 2c$.
  Finally, it follows that
  \begin{align*}
    \begin{pmatrix}
      u(1) \\ u'(1)
    \end{pmatrix}
    &= -\frac{1}{\gamma} \begin{pmatrix}
      1 & 1 \\ \gamma_1 & \gamma_2
    \end{pmatrix}
      \begin{pmatrix}
      \gamma_2 \xi^{1-\gamma_2}  & -\xi^{-\gamma_2} \\ -\gamma_1\xi^{1-\gamma_1} &
      \xi^{-\gamma_1}
    \end{pmatrix}
    \begin{pmatrix}
      u(\xi^{-1}) \\ u'(\xi^{-1})
    \end{pmatrix} \\
    & \approx - u(-1) \begin{pmatrix}
      1 & 1 \\ \gamma_1 & \gamma_2
    \end{pmatrix}
                          \begin{pmatrix}
                            (2\gamma^2+2c^2) \xi^{1-2\gamma_2} \\
                            -1 \xi^{0}
                          \end{pmatrix}\\
    &\approx u(-1) \xi^{1-2\gamma_2}
      \begin{pmatrix}
        8c^2 \\ 8c^2
      \end{pmatrix}.
  \end{align*}
\end{proof}

We thus see that, while the evolution on the small interval
$(-\xi^{-1},\xi^{-1})$ yields a singular value of size $\xi^{1}$, the
conjugation with the power law evolution on $|t|>\xi^{-1}$ yields a much smaller
singular value $\xi^{\gamma}\ll \xi^{1}$.

Having establish a precise description of the evolution for times close to a
single resonant time $\tau \approx \frac{1}{l}$, in the following we consider an
iterated model to study the norm inflation in Gevrey spaces and the associated
asymptotic behavior.

\subsection{Model Echo Chains and Modified Exponents}
\label{sec:2mode}

In Section \ref{sec:chains} we have studied chains of echoes for the linear
problem \eqref{eq:linearmodel} and have established norm inflation with a factor
$\exp(C \sqrt{\eta})$. However, in that case our proof limited us to considering
only $\eta$ such that $c \ln(1+\eta^2)$ is not too large.

In the following we instead consider an iterated version of the model of Section
\ref{sec:Model}, which does not possess such an obstruction.
In particular, combining the behavior of infinitely many modes $\eta_j$ with
$\eta_j\rightarrow \infty$ we construct solutions which exhibit norm inflation
for arbitrarily large times and do not converge as time tends to infinity.
However, despite the failure of the convergence of the vorticity, the velocity
field is shown to converge.

We briefly recall the approximations made in the preceding sections of this
article.
We started with the 2D Euler equations close to Couette flow
\begin{align*}
  \dt \omega + y\p_x \omega + v \cdot \nabla \omega =0,
\end{align*}
and focused on perturbations of the form $\omega= 2c \cos(x+ty)+ \epsilon
\omega_*$.
Omitting the transport by $2c\frac{\sin(x+ty)}{1+t^2}$ and changing to variables
$(x+ty,y)$, we obtain
\begin{align*}
  \dt \omega + 2 c\sin(x) \p_y \Delta_t^{-1} \omega + \epsilon \nabla^\bot \Delta_t^{-1} \omega  \cdot \nabla \omega=0. 
\end{align*}
In particular, we note that this equation formally conserves $\|2c \cos(x)+
\epsilon \omega\|_{L^2}$.
Omitting the nonlinearity by letting $\epsilon \downarrow 0$, we loose this
conserved quantity, but obtain an explicit Fourier problem with
nearest-neighbor-interactions:
\begin{align}
  \label{eq:7}
  \begin{split}
  \dt \tilde{\omega}(t,k,\eta) &+ 2c \frac{\eta}{(k-1)^2+(\eta-(k-1)t)^2} \tilde{\omega}(t,k-1,\eta) \\ &- 2c \frac{\eta}{(k+1)^2+(\eta-(k+1)t)^2} \tilde{\omega}(t,k+1,\eta)=0.
  \end{split}
\end{align}
This further highlights resonant times, where $\eta-(k\pm 1)t\approx 0
\Leftrightarrow t \approx \frac{\eta}{k \pm 1}$.
After studying the system \eqref{eq:7} in Sections \ref{sec:LWP} and
\ref{sec:chains}, in Section \ref{sec:Model} we further introduced a model
problem \ref{eq:twomodesystem} that focuses solely on resonant modes $v$ ($(l,\eta)$ such that $\eta-lt
\approx 0$) and their neighbors u:
\begin{align}
  \begin{split}
  \p_\tau u - \frac{c}{\eta^{-2}+\tau^2} \frac{1}{l^2} v(\tau) =0, \\
  \p_\tau v + \frac{c}{\eta^{-2}+l^{-4}} \frac{1}{l^2} u(\tau) =0.
  \end{split}
\end{align}
In particular, we showed that $v|_{\tau=\tau_1}$ is approximately of size $c
(\frac{\eta}{l^2})^\gamma u|_{\tau=\tau_0}$.\\

Building on the single resonance results of Section \ref{sec:Model}, we
construct the following iteration scheme:
\begin{itemize}
\item Let $k \in \N$, $k>1$ and $\eta \in \R$ be given and define
  $\tau_k=\frac{1}{2}(\frac{1}{k-1}+\frac{1}{k})$.
  We then prescribe $(u,\p_\tau u)|_{\tau_k}$ and use equation \eqref{eq:schroedinger2}
  to determine $(u,\p_t u)|_{\tau_{k-1}}$.
\item Relabeling $v$ of the previous step as $u$ of the case $k-1$ and using \eqref{eq:2} we prescribe
  \begin{align*}
    \begin{pmatrix}
      u_{k-1} \\ u'_{k-1}
    \end{pmatrix}|_{\tau=\tau_{k-1}}
:= 
    \begin{pmatrix}
      c^{-1} \ u_{k}' \\ c u_k
    \end{pmatrix}|_{\tau=\tau_{k-1}}.
  \end{align*}
  We then again use equation \eqref{eq:schroedinger2} to determine
  $(u,u')|_{\tau_{k-2}}$. 
\item We iterate this procedure until we reach $\tau_1$, where we define $\tau_0=1.5$.
\end{itemize}
Recalling the construction of the model problem of Section \ref{sec:Model},
$(u,\p_t u)_{\tau=\tau_k}$ corresponds to prescribing $(\tilde{\omega}(t,k,
\eta), \tilde{\omega}(t,k-1,\eta))$ at time
$t=\frac{1}{2}(\frac{\eta}{k-1}+\frac{\eta}{k})$  and $(u,u')|_{\tau=\tau_0}$
corresponds to the value of the modes $(1,\eta)$ and $(0,\eta)$ at time $t=1.5 \eta$.

\begin{thm}
  \label{thm:modelgrowth}
  Let $k \in \N$ and $\eta \in \R$ be given and prescribe $(u,\p_t
  u)|_{\tau=\tau_k}=(1,0)$.
  Then, the above iteration scheme yields that
  \begin{align*}
    (u, \p_t u) |_{\tau=\tau_0} \approx c^{k}\left(\frac{\eta^k}{(k!)^2}\right)^\gamma,
  \end{align*}
  where $\gamma=\sqrt{1-4c^2}\neq 1$.
  In particular, choosing $k$ maximally for $c,\eta$ fixed, we obtain a growth
  factor consistent with a Gevrey regularity class.
\end{thm}

\begin{proof}
  Using the result of Theorem \ref{thm:specialfunctions}, we observe that
  \begin{align*}
    \begin{pmatrix}
      u_{k-1} \\ u'_{k-1}
    \end{pmatrix}|_{\tau=\tau_{k-1}}
    \approx
    u_k'(\tau_{k})\xi^{\gamma}
    \begin{pmatrix}
      8c  \\ 8c^3
    \end{pmatrix}.
  \end{align*}
  In particular, we note that $u_{k-1}\approx c^{-2}u'_{k-1}$ and we may thus
  apply Theorem \ref{thm:specialfunctions} again. We repeat this process another $k-1$ times, where $\xi=\frac{\eta}{k^2}$ changes in each step, and thus obtain the claimed growth factor
  \begin{align*}
    c^k \left(\frac{(c\eta)^k}{(k!)^2}\right)^{\gamma}.
  \end{align*}
  Considering $c\eta^{\gamma}$ large and $k$ large, by Stirling's approximation it holds that
  \begin{align*}
   c^{k} \left(\frac{\eta^{k}}{(k!)^2}  \right)^{\gamma} \sim c^{k} \left(\eta^{k} \frac{e^{2k}}{2\pi k k^{2k}}  \right)^{\gamma}.
  \end{align*}
  Choosing $k=\sqrt{c^{\frac{1}{\gamma}}\eta}$, we obtain a cancellation of
  $c^{k }\left(\frac{\eta^{k}}{k^{2k}}\right)^{\gamma}=1$ and thus
  \begin{align*}
    e^{2\gamma \sqrt{c^{1/\gamma}\eta}} \left(\frac{1}{2\pi \sqrt{c^{1/\gamma}\eta}}  \right)^\gamma
  \end{align*}
  as the maximal growth factor.
\end{proof}

As a corollary, for this model we can construct initial data in a critical
Gevrey regularity class.
\begin{thm}
  \label{thm:toyscattering}
  Consider the chained two mode model with $0<c<\frac{1}{2}$ and
  $\gamma=\sqrt{1-4c^2}$. Then there exists $s=s(\gamma)$ and $C>0$ such that for
  every $\epsilon>0$ and every $\sigma_0 \in \R$, there exists initial data $u_0 \in \mathcal{G}_{s,C}$ such
  that
  \begin{align}
    \|u_0\|_{\mathcal{G}_{s,C}}< \epsilon,
  \end{align}
  and such that for every $\tilde{C}>0$,
  \begin{align*}
    \lim_{t \rightarrow \infty} \|u(t)\|_{\mathcal{G}_{s,\tilde{C}}}=\infty.
  \end{align*}
  Furthermore, $u(t)$ does converge in $H^{\sigma}$, $\sigma<\sigma_0$, but
  diverges in $H^{\sigma}, \sigma>\sigma_0$.
\end{thm}
In particular, choosing $-1< \sigma_0 <0$, we find initial data, arbitrarily
small in the critical Gevrey regularity class, such that we do not converge in
$L^2$ as $t \rightarrow \infty$, but the velocity field does converge.

\begin{proof}
  For any given $\eta$, let $k_\eta$ be the associated maximizer of the growth
  factor obtained in Theorem \ref{thm:modelgrowth} and let $g(\eta)$ denote that
  growth factor.
  
  Let now $\psi \in \cap_{\sigma<\sigma_0}H^{\sigma}(\R)$ be given and prescribe as
  initial data
  \begin{align*}
    u= \int_{\eta} \frac{1}{g(\eta)} \hat{\psi}(\eta) e^{i \eta y + i k_{\eta}x}.
  \end{align*}
  Since $g(\eta)\approx \exp(C \sqrt{\eta})$ for $C=C(c)$, this function is in a
  Gevrey class.
  
  Then by construction, the $k=1$ mode will asymptotically be given by
  \begin{align*}
    \int_{\eta} g(\eta) \frac{1}{g(\eta)} \hat{\psi}(\eta) e^{i\eta y + i x}= \psi(y) e^{ix}
  \end{align*}
\end{proof}
Building on the insights obtained in this model problem in the following Section
we consider the full problem.

\section{Echo chains as a linear mechanism and modified scattering}
\label{sec:echo}
In the previous section we have shown that for large $\eta$ a linear growth
factor $\frac{\eta}{l^2}$ cannot be expected to be accurate anymore. Indeed, the
logarithmic corrections in the Duhamel iteration are much larger than the prior
``leading term''.
Instead we expect to see a modified exponent, which is less than $1$ due to
cancellations with neighboring modes.
In the following we consider the full model \eqref{eq:tau}
\begin{align}
  \begin{split}
  \p_\tau \omega (\tau,k,\eta) &+ c \frac{1}{(k-1)^2} \frac{1}{\eta^{-2} +(\frac{1}{k-1}-\tau)^2} \omega(\tau,k-1,\eta) \\ &- c \frac{1}{(k+1)^2} \frac{1}{\eta^{-2} +(\frac{1}{k+1}-\tau)^2} \omega(\tau,k+1,\eta)=0.
  \end{split}
\end{align}
on a time-interval around a single resonance $\frac{1}{k_0}$,
\begin{align*}
  \left(\frac{\frac{1}{k_0+1} + \frac{1}{k_0}}{2}, \frac{\frac{1}{k_0}+\frac{1}{k_0-1}}{2}\right).
\end{align*}
Considering that
\begin{align*}
\frac{1}{k_0\pm 1}- \frac{1}{k_0}= \frac{1}{k_0(k_0\pm 1)} \leq \frac{2}{k_0^2},
\end{align*}
and the central role of $\tau=\frac{1}{k_0}$, compared to Section
\ref{sec:Model} we again change variables as
\begin{align}
  t= k_0^{2}(\tau-\frac{1}{k_0}) \in \left(-\frac{k_0}{2(k_0+1)},\frac{k_0}{2(k_0-1)}\right)=:(t_0,t_1).
\end{align}
Then $\p_\tau= k_0^2 \p_t$ and hence \eqref{eq:tau} reads
\begin{align}
  \label{eq:13}
  \dt \omega(t,k,\eta) + a(k-1) \omega(t,k-1,\eta)-  a(k+1) \omega(t,k+1,\eta)=0,
\end{align}
where
\begin{align}
  a(k_0)= \frac{c}{(\frac{\eta}{k_0^2})^{-2}+t^2}
\end{align}
and we abbreviate
\begin{align}
  \xi:= \frac{\eta}{k_0^2},
\end{align}
and for $k\neq k_0$
\begin{align}
  \label{eq:14}
  a(k)&= \frac{1}{k^2k_0^2} \frac{c}{\eta^{-2} + (\frac{1}{k}-\frac{1}{k_0}-k_0^{-2}t)^2}\\
  &= \frac{c}{\eta^{-2}k_0^2 k^2 + (k_0-k -\frac{k}{k_0}t)^2}.
\end{align}
In particular, we note that since we only consider the resonant interval around
$k_0$ and are thus far from other resonant times $|a(k)|\leq 4 c$ for any $k\neq
k_0$.
In contrast, $a(k_0)$ at time $t=0$ is of size $c \xi^{2} \gg 1$.

We remark that in Section \ref{sec:Model} we made several simplifications
compared to the full model:
\begin{enumerate}
\item We approximated $a(k_0\pm 1)\approx \pm c$, which allowed us to compute
  explicit solutions. In the following we need to show that
  this is a valid approximation, that is the evolution of the full problem can
  be estimated above and below by the approximate evolution.
\item In Section \ref{sec:Model} we considered a two-mode model involving just
  $k_0$ and $k_0-1$. Instead we show that the precise behavior is more
  accurately captured by the three-mode model involving $k_0-1,k_0,k_0+1$, which
  yields a change of the exponent $\gamma$ (involving $2c^2$ in place of $c^2$). 
\item In view of the sizes of $a(k)$ in the following Section \ref{sec:setup} we
  at first again neglect all except the three modes $k_0-1,k_0,k_0+1$. We call
  this the \emph{three-mode model}. In contrast to the two-mode model of Section
  \ref{sec:2mode} we here do not approximate the coefficient functions. In Section \ref{sec:inhomogeneous} we then
  discuss the full problem incorporating all modes and prove that indeed all
  other modes can treated as perturbations in a bootstrap approach.  
\end{enumerate}

\subsection{The Three-mode Model}
\label{sec:setup}

In this section we introduce the solution operator of the homogeneous three-mode
model:
\begin{align}
  \label{eq:homogeneous}
  \dt u(k)+ a(k+1)u(k+1)-a(k-1)u(k-1)&=0,\\
a(k)&=
  \begin{cases}
    a(k_0\pm 1)= \frac{c}{\xi^{-2}(\frac{k_0\pm 1}{k_0})^2 + (1\pm \frac{k_0\pm 1}{k_0}t)^2}, \\
    a(k_0)= \frac{c}{\xi^{-2}+t^2},\\
0.
  \end{cases}
\end{align}
That is, we consider only the modes $k_0-1, k_0, k_0+1$ and neglect all other
modes as ``inhomogeneities''. This allows for a clearer
discussion of the growth and decay mechanisms and serves to introduce the
techniques of proof used in the different regimes.
In Section \ref{sec:inhomogeneous} we then show that a similar behavior also
holds in the full model.

Here the heuristic of the approximate model $a(k_0\pm 1)\approx c$ suggests a
power law behavior of solutions. As for the present case of exact
coefficients an explicit solution is not feasible anymore, in the following we
establish a comparison estimate. We argue in multiple steps:
\begin{itemize}
\item By symmetry it holds that $u(k_0+1)+u(k_0-1)=\text{const.}$. We hence to
  some extent reduce to a two-mode model. However, as $a(k_0+1)\neq a(k_0-1)$ in
  this model the problem does not completely decouple.  
\item We first establish a rough power law upper bound on the growth of
  solutions on $(-t_0,-t)$ as $t \downarrow 0$. This is achieved by
  concatenating multiple small time estimates.
\item Subsequently, we iteratively improve this bound to $t^{\gamma_2-1}$
  upper and lower bounds similar as in Section \ref{sec:Model} (but with $c^2$
  replaced by $2c^2$ due to the third mode). This step relies on
  reformulations of the ODE system as second order ODEs and integrating these.
\item We then show that the resonant mode is decreasing as $t^{\gamma_2}$ while
  the neighboring modes grow like $t^{\gamma_2-1} $both
  with an upper and lower bound.
\item Combining these results, we construct solution operators on the intervals
  $I_1=(t_0,-\frac{d}{\xi})$, $I_3=(\frac{d}{\xi},t_1)$. On the resonant
  interval $I_2=(-\frac{d}{\xi}, \frac{d}{\xi})$ we instead use a Duhamel
  iteration argument to construct the solution operator. 
\item Concatenating the solution operators we obtain the solution operator from
  time $t_0$ to $t_1$ and show that it exhibits analogous $(\frac{\eta}{k_0^2})^{\gamma}$ growth
  behavior to the model problem of Section \ref{sec:Model}. In Section
  \ref{sec:inhomogeneous} we then show that this behavior persists in the full problem. 
\end{itemize}

The main results of this section are summarized in the following theorem.

\begin{thm}[The full solution, homogeneous case]
  \label{thm:homogeneous_full}
  Let $\xi=\frac{\eta}{k_0^2} \gg 1$ and $0<c< 0.2$ be given.
  We consider the ODE system \eqref{eq:homogeneous} for $u_1= \frac{u(k_0+1)-u(k_0-1)}{2}, u_2=u(k_0),
  u_3=\frac{u(k_0+1)+u(k_0-1)}{2}$:
  \begin{align}
    \label{eq:8}
    \dt u + M(t) u =0,
  \end{align}
  where
  \begin{align*}
    M(t)=
\begin{pmatrix}
  0 &  a(k_0) & a(k_0+1)-a(k_0-1)\\
  -(a(k_0+1)-a(k_0-1)) & 0 & -a(k_0+1)+a(k_0-1) \\
  0 & 0 & 0
    \end{pmatrix}
  \end{align*}
  on the interval $(t_0,t_1)$.

  Suppose that at time $t_0=-\frac{1}{2}\frac{k_0}{k_0+1}$ it holds that $u(k_0) \geq 0.5 \max(u)$.
  Then at time $t_1=\frac{1}{2}\frac{k_0}{k_0-1}$ it holds that
  \begin{align*}
    u_1(t_1) &\approx c^{2-2\gamma_2} \xi^{\gamma} u_2(t_0),\\
    u_2(t_1)&\approx c^{4-2\gamma_2}\xi^{\gamma} u_2(t_0),\\
    u_3(t_1) &= u_3(t_0), 
  \end{align*}
  where $\gamma_2=\frac{1}{2}-\sqrt{\frac{1}{4}-2c^2}$ and $\gamma=1-2 \gamma_2=\sqrt{1-8c^2}$.
\end{thm}
We in particular note that the exponent here is different from $1$ (which was
not visible in prior works due to the logarithmic constraints) and that at time $t_1$ our solution $u$ satisfies the
assumptions of this theorem with $k_0$ replaced by $k_0-1$. Thus, we may
iteratively apply this theorem until we reach the frequency $1$ and obtain the
following corollary. In Section \ref{sec:mod} we return to this echo chain
behavior in the context of the full problem and also discuss (modified)
asymptotic behavior as $t\rightarrow \infty$.

\begin{cor}[Echo chain, homogeneous case]
  \label{cor:homogeneous_chain}
  Let $\eta \gg 1$ and $k_0\ll \eta$ be given.
  We then consider the iterated system where initially set
  $\xi=\frac{\eta}{k_0^2}$ and solve \eqref{eq:8} with
  \begin{align*}
    u(t_0)=
    \begin{pmatrix}
      0 \\ 1 \\0
    \end{pmatrix}
  \end{align*}
  In the next step, we then decrease $k_0 \mapsto k_0-1$ and set
  $\xi=\frac{\eta}{(k_0-1)^2}$ and again solve \eqref{eq:8} with new initial
  data being given by
  \begin{align*}
    \begin{pmatrix}
      u_{1}(t_0)\\
      u_{2}(t_0)\\
      u_{3}(t_0)
    \end{pmatrix}_{new}
    =
    \begin{pmatrix}
      0\\
      u_{1}(t_1)\\
      0
    \end{pmatrix}_{old}
  \end{align*}
  We iterate this procedure another $k_0-2$ times until we have reached $k=1$.
  At this time it then holds that
  \begin{align*}
    u_1(1)\approx (c^{2-2\gamma_2})^{k_0} \left(\frac{\eta^{k_0}}{(k_0!)^2}  \right)^{\gamma}.
  \end{align*}
  In particular, choosing $k_0$ maximally for given $\eta$ we obtain that
  $u_1(1)\approx C_1 \exp(C_2 \sqrt{\eta})$ attains a Gevrey $2$ norm inflation factor.
\end{cor}
We remark that the precise choice of $u_1,u_3$ in the update step here does not change
the result as long as $u_2(t_0)$ is (comparable to) the largest entry and thus
Theorem \ref{thm:homogeneous_full}  can be applied.

\begin{proof}[Proof of Corollary \ref{cor:homogeneous_chain}]
Let $k_0$ be given and let $t_{0}[k_0]=\frac{1}{2}(\frac{1}{k_0+1}-\frac{1}{k_0})$,
$t_1[k_0]=\frac{1}{2}(\frac{1}{k_0-1}-\frac{1}{k_0})$.
Then at time $t_0[k_0]$ it holds that $u_2 \geq 0.5 \max(u)$ and we may thus
apply Theorem \ref{thm:homogeneous_full} to conclude that at time $t_1[k_0]$,
\begin{align*}
  u_1(t_1[k_0]) \approx c^{2-2\gamma_2} \xi^{\gamma} u_2(t_0[k_0]) \geq 0.5 \max (u(t_1[k_0])).
\end{align*}
We then decrease $k_0$ by $1$ and obtain that at the new initial time
$t_0[k_0-1]=t_1[k_0]$ by relabeling the above
\begin{align*}
  u_2(t_0[k_0-1])= u_1(t_1[k_0]) \geq 0.5 \max(u)
\end{align*}
again satisfies the assumptions of Theorem \ref{thm:homogeneous_full}.
Iterating this procedure until we reach $k_0=1$ then yields the result.
\end{proof}

Our proof proceeds by considering the three intervals $I_1=(t_0,
-\frac{d}{\xi})$, $I_2=(-\frac{d}{\xi}, \frac{d}{\xi})$, $I_3=(\frac{d}{\xi},
t_1)$, where $d=c^{-1}$.
The right boundary datum of each interval then serves as the left boundary datum
of the next.

In the following Section \ref{sec:upperandasymptotics} we derive upper and lower
bounds for the evolution on the intervals $I_1$ and $I_3$.
In Section \ref{sec:homogeneousintervals} we then combine these bounds and a
characterization of the evolution on $I_2$ to conclude our proof of Theorem \ref{thm:homogeneous_full}.

\subsubsection{The interval $I_1$}
\label{sec:upperandasymptotics}
Our main estimates for the evolution of the homogeneous problem
\eqref{eq:homogeneous} on the interval $I_1=(t_0,-\frac{d}{\xi})$ are summarized in the following proposition.

\begin{prop}[Left interval, homogeneous case]
  \label{prop:lefthom}
  Let $\xi=\frac{\eta}{k} \gg 1$ and $0<c <0.2$ be given.

  We consider the problem \eqref{eq:homogeneous} on the interval $I_1=(t_0,-\frac{d}{\xi})$ with $d=c^{-1}$.
  Then the unique solution $u(t)$ satisfies
  \begin{align*}
    u_3(t)&=u_3(t_0)
  \end{align*}
  Furthermore, it holds that
  \begin{align*}
    |u_1(t)| &\leq C t^{\gamma_2-1} \max (u(t_0)), \\
    |u_2(t)| &\leq C t^{\gamma_2} \max (u(t_0)).
  \end{align*}
  where $\gamma_1=\frac{1}{2}+\sqrt{\frac{1}{4}-2c^2}$,
  $\gamma_2=\frac{1}{2}-\sqrt{\frac{1}{4}-2c^2}$.

  If in addition  $|u_2(t_0)| \geq 4c |u(t_0)|$, then at time
  $t=-\frac{d}{\xi}$ it holds that
  \begin{align*}
    u_1(-\frac{d}{\xi})&\approx (\frac{d}{\xi})^{\gamma_2-1}\frac{c}{\gamma_2-\gamma_1} u_2(t_0), \\
    u_2(-\frac{d}{\xi})&\approx (\frac{d}{\xi})^{\gamma_2} \frac{\gamma_1}{\gamma_2-\gamma_1}u_2(t_0).
  \end{align*}
For later reference we note that $u_1(-\frac{d}{\xi})\gg u_2(-\frac{d}{\xi})$ and that
$u_1(-\frac{d}{\xi})(\frac{d}{\xi})\approx  c \ u_2(-\frac{d}{\xi})$. 
\end{prop}

We may write the homogeneous three-mode model \eqref{eq:homogeneous} as
\begin{align}
  \label{eq:homogeneous1}
  \dt
  \begin{pmatrix}
    u(k_0-1) \\ u (k_0) \\ u(k_0+1)
  \end{pmatrix}
+
  \begin{pmatrix}
    0 & a & 0 \\
    - b_1 & 0 & b_2 \\
    0 & -a & 0
  \end{pmatrix}
  \begin{pmatrix}
    u(k_0-1) \\ u (k_0) \\ u(k_0+1)
  \end{pmatrix}
             =0,
\end{align}
where $a=\frac{c}{\xi^{-2}+t^2}$, $b_1\neq b_2\approx c$.
For this structure it is apparent that $u_3:=\frac{1}{2}(u(k_0-1)+u(k_0+1))$ is
conserved.

\begin{lem}[Reduction 1]
  \label{lem:reduction1}
  Consider the problem \eqref{eq:homogeneous1}, then it holds that
  \begin{align*}
    u(k-1) + u(k+1) =\text{const.}.
  \end{align*}
  We hence introduce the change of unknowns and notation
  \begin{align*}
    u_1&= \frac{1}{2}(u(k_0-1)-u(k_0+1)), \\
    u_2&= u(k_0), \\
    u_3&= \frac{1}{2}(u(k_0-1)+u(k_0+1)),
  \end{align*}
  where $u_3$ is invariant under the evolution and $(u_1,u_2)$ solve
  \begin{align}
    \label{eq:twomode1}
    \dt
    \begin{pmatrix}
      u_1 \\ u_2
    \end{pmatrix}
+
    \begin{pmatrix}
      0 & a \\
      -b & 0
    \end{pmatrix}
          \begin{pmatrix}
            u_1 \\ u_2
          \end{pmatrix}
  =
    \begin{pmatrix}
      (b_1+b_2)u_3 \\
      (b_1-b_2)u_3.
    \end{pmatrix}
,
  \end{align}
  where $b=b_1+b_2$.
\end{lem}

\begin{proof}[Proof of Lemma \ref{lem:reduction1}]
  We observe that
  \begin{align*}
    \dt (u(k_0+1)+u(k_0-1))= -a u(k_0)+ a u(k_0)=0
  \end{align*}
  and hence $u_3$ is conserved.
  The equation satisfied by $u_1,u_2$ is then just a reformulation of equation \eqref{eq:homogeneous1}.
\end{proof}

In order to solve \eqref{eq:twomode1} we first focus on the special case $u_3=0$,
\begin{align}
  \label{eq:9}
      \dt
    \begin{pmatrix}
      u_1 \\ u_2
    \end{pmatrix}
+
    \begin{pmatrix}
      0 & a \\
      -b & 0
    \end{pmatrix}
          \begin{pmatrix}
            u_1 \\ u_2
          \end{pmatrix}
  =0.
\end{align}
That is, we study the homogeneous solutions of \eqref{eq:twomode1}.
In Lemma \ref{lem:particular_solution} we then construct a particular solution
for the case $u_3\neq 0$.

As a first heuristic, note that on $I_1=(t_0,-\frac{d}{\xi})$ it seems
reasonable to approximate
\begin{align}
  a(t)&= \frac{c}{\xi^{-2}+t^2} \approx \frac{c}{t^2}, \\
  b(t)&\approx 2c.
\end{align}
This approximated problem can be explicitly solved and suggests that
\begin{align}
  \begin{pmatrix}
    u_1 \\ u_2
  \end{pmatrix}
  \approx
  \begin{pmatrix}
    \frac{\gamma_1}{c} |t|^{\gamma_1-1} & \frac{\gamma_2}{c} |t|^{\gamma_2-1} \\
    |t|^{\gamma_1} & |t|^{\gamma_2}
  \end{pmatrix}
                     \begin{pmatrix}
                       \alpha \\ \beta
                     \end{pmatrix},
\end{align}
for suitable constants $\alpha, \beta$ and $\gamma_1=\frac{1}{2}+
\frac{1}{2}\sqrt{1-8c^2}$, $\gamma_2=\frac{1}{2}+
\frac{1}{2}\sqrt{1-8c^2}$.
In the following we will show that this heuristic is indeed valid in the sense
that the actual solution operator has the same asymptotic power law behavior as $|t|$ becomes small.
\\

In order to establish these asymptotics we make use of a self-improving
estimate. That is we will assume for the moment that there exists some
$\sigma<\infty$ such that
\begin{align}
  \label{eq:roughupper}
  |u_1(t)| \leq C |t|^{-\sigma} |u(t_0)|,
\end{align}
and show that then $u_1$ and $u_2$ necessarily also satisfy the upper lower
bounds given in Proposition \ref{prop:lefthom}.
This a priori assumption is established at the end of this Section in Lemma \ref{lem:roughupperbound}.

\begin{lem}[Improvement]
  \label{lem:improvement}
  Consider the problem \eqref{eq:9} on the interval $I_1=(t_0,-\frac{d}{\xi})$
  and suppose that $u_1$ satisfies the upper bound \eqref{eq:roughupper}.

  Then there exists a constant $C>0$ (independent of $\xi, \epsilon$, possibly dependent on $c$),
  such that
  \begin{align}
    |u_1(t)| \leq C |t|^{\gamma_2-1} |u(t_0)|, \\
    |u_2(t)| \leq C |t|^{\gamma_2} |u(t_0)|,
  \end{align}
  where $\gamma_2= \frac{1}{2}-\frac{1}{2}\sqrt{1-8c^2}$.
  
  Furthermore, if $|u_2(t_0)|\geq 0.5 |u(t_0)|$, then at time $t=-\frac{d}{\xi}$
  it holds that 
  \begin{align}
    \label{eq:15}
    \begin{split}
    |u_1(t)| \geq C/2 (\frac{d}{\xi})^{\gamma_2-1}|u_2(t_0)|, \\
    |u_2(t)| \leq C/2 (\frac{d}{\xi})^{\gamma_2}|u_2(t_0)|. 
    \end{split}
  \end{align}
  An analogous result holds on the interval $I_3=(\frac{d}{\xi},t_1)$.
\end{lem}

\begin{proof}[Proof of Lemma \ref{lem:improvement}]
We note that $|b(t)-2c| \leq  10 ct$ by Taylor's formula and hence treat it as a
perturbation.

\underline{A toy model:}
In order to introduce our method of proof let us first consider the approximated
problem where we replace $a(t)$ by exactly $\frac{c}{t^2}$.
Then our problem can be written as
    \begin{align*}
    \dt u +
    \begin{pmatrix}
      0 & \frac{c}{t^2} \\
      -2c & 0
    \end{pmatrix}
            u =
            \begin{pmatrix}
              0 \\ \mathcal{O}(c|t|) u_1
            \end{pmatrix},
    \end{align*}
 where by assumption \eqref{eq:roughupper}it holds that $\mathcal{O}(c|t|)u_1=
 c\mathcal{O}(|t|^{1-\sigma})$. We then consider a variation of constants
 ansatz:
  \begin{align*}
    u(t)=
    \begin{pmatrix}
      \frac{\gamma_1}{c}|t|^{\gamma_1-1} & \frac{\gamma_2}{c}|t|^{\gamma_2-1}\\
      |t|^{\gamma_1} & |t|^{\gamma_2}
    \end{pmatrix}
                       \begin{pmatrix}
                         \alpha \\ \beta
                       \end{pmatrix}.
  \end{align*}
Plugging this in, we get that
  \begin{align*}
        \begin{pmatrix}
      \frac{\gamma_1}{c}|t|^{\gamma_1-1} & \frac{\gamma_2}{c}|t|^{\gamma_2-1}\\
      |t|^{\gamma_1} & |t|^{\gamma_2}
    \end{pmatrix} \dt \begin{pmatrix}
                         \alpha \\ \beta
                       \end{pmatrix} =
    \begin{pmatrix}
      0 \\
      \mathcal{O}(ct) u_1(t)
    \end{pmatrix}\\
   \Leftrightarrow \dt \begin{pmatrix}
      \alpha \\ \beta
    \end{pmatrix} = \frac{c}{\gamma}
    \begin{pmatrix}
      |t|^{\gamma_2} & - \frac{\gamma_2}{c}|t|^{\gamma_2-1} \\
      - |t|^{\gamma_1} & \frac{\gamma_1}{c}|t|^{\gamma_1-1} 
    \end{pmatrix}
    \begin{pmatrix}
      0 \\
      \mathcal{O}(ct) u_1(t)
    \end{pmatrix}.
  \end{align*}
  Now using that $\mathcal{O}(ct)|u_1|\leq Cc |t|^{1-\sigma}$, we get that
  \begin{align*}
    |\dt \alpha| \leq Cc |t|^{\gamma_2-1} |t|^{1-\sigma}= Cc |t|^{\gamma_2-\sigma}, \\
    |\dt \beta|\leq  Cc |t|^{\gamma_1-\sigma}.
  \end{align*}
  Integrating, we obtain that
  \begin{align*}
    |\alpha|\leq C + C|t|^{1+\gamma_2-\sigma},\\
    |\beta|\leq C + C|t|^{1+\gamma_1-\sigma}.
  \end{align*}
  Plugging this into our ansatz it follows that
  \begin{align*}
    |u_1(t)|&\leq |t|^{\gamma_1-1} (C +|t|^{1+\gamma_2-\sigma} ) + |t|^{\gamma_2-1} (C + |t|^{1+\gamma_1-\sigma}) \\
            &\leq  C |t|^{\gamma_2-1} + C |t|^{1-\sigma}, \\
    |u_2(t)| & \leq |t|^{\gamma_1} (C +|t|^{1+\gamma_2-\sigma} ) + |t|^{\gamma_2} (C+ |t|^{1+\gamma_1-\sigma} ) \\
    & \leq C |t|^{\gamma_1} + C |t|^{2-\sigma} + C |t|^{2-\sigma}.
  \end{align*}
  Here, we used that $\gamma_1+\gamma_2=1$, but may also more roughly estimate $\gamma_1+\gamma_2\geq 2 \gamma_2 >0$.
  
  In particular, we observe that
  \begin{align*}
    |u_1(t)| \leq  C |t|^{\gamma_2-1} + C |t|^{2\gamma_2-\sigma} \leq C |t|^{-\sigma}
  \end{align*}
  is a strict improvement over \eqref{eq:roughupper} if $\sigma> 1-\gamma_2$
  initially.
  We may then repeat our argument with \eqref{eq:roughupper} with this smaller
  $\sigma$ and successively improve until $\sigma=1-\gamma_2$.\\

  With this improved upper bound at hand, we now can establish the lower bounds
  \eqref{eq:15}. For this we observe that
  \begin{align*}
    \begin{pmatrix}
      \alpha(t_0) \\ \beta(t_0)
    \end{pmatrix}
   & =
    S(t_0)^{-1}
    \begin{pmatrix}
      u_1(t_0) \\ u_2(t_0)
    \end{pmatrix} \\
&\approx  \frac{c}{\sqrt{1-8c^2}}
     \begin{pmatrix}
       \gamma_1/c t_0^{\gamma_2-1} & - t_0^{\gamma_2} \\
       - \gamma_2/c t_0^{\gamma_1-1} & t_0^{\gamma_1}
     \end{pmatrix}
                    \begin{pmatrix}
                              u_1(t_0) \\ u_2(t_0)
                            \end{pmatrix}.
  \end{align*}
  Under the assumption of the lower bound on $u_2(t_0)$ it thus follows that
  $\beta(t_0)$ is bounded below and of comparable size to $u_2(t_0)$.
  It thus only remains to show that $(\alpha(t),\beta(t))$ do not deviate too
  much from this.
  Here, we use the above developed bounds on $\dt
  \alpha$ and $\dt \beta$:
  \begin{align*}
    |\dt \beta| &\leq \frac{c}{\gamma}  \frac{\gamma_1}{c} |t^{\gamma_1}|
                  |u_1(t)| |t| Cc , \\
    \leadsto |\beta(t)-\beta(t_0)| &= \mathcal{O}(c), \\
    |\dt \alpha|&\leq \frac{c}{\gamma} \frac{\gamma_2}{c} |t|^{\gamma_2-1} Cc |t| |u_1(t)|\\
    \leadsto |\alpha(t)-\alpha(t_0)|&= \mathcal{O}(c^3),
  \end{align*}
  where we used that $\gamma_2=\frac{1}{2}-\sqrt{\frac{1}{4}-2
    c^2}=\mathcal{O}(c^2)$. Since $c$ is small, $\alpha(t),\beta(t)$ hence
  indeed do not deviate too much and the result follows.
\\

  \underline{Proof in the actual model:} This proof is analogous with the only
  difference being that the solution matrix is now not anymore given by power
  laws but rather by explicit hypergeometric/Legendre functions, which are bounded above
  and below by such power laws.

  We note that the problem
  \begin{align*}
    \dt \psi +
    \begin{pmatrix}
      0 & \frac{c}{\xi^{-2}+t^2} \\
      -2c & 0
    \end{pmatrix}
\psi=0,
  \end{align*}
  can be decoupled into two second order equations and that $\psi_1$ solves
  \begin{align*}
    \dt (\xi^{-2}+t^2) \dt \psi_1 + 2c^2 \psi_1=0.
  \end{align*}
  After some changes of coordinates one sees that this is a Legendre type
  equation
  \begin{align*}
    (1-t^2)y''(t) -2t y'(t)+ \nu(\nu+1)y(t)=0
  \end{align*}
  for imaginary $t$, which has the general solution
  \begin{align*}
    \psi_1(t)= C_1 P(-\frac{1}{2}+\frac{1}{2}\sqrt{1-8c^2}, it \xi)
    + C_2 Q(-\frac{1}{2}+\frac{1}{2}\sqrt{1-8c^2}, it\xi),
  \end{align*}
  where $P$ is the Legendre function of the first kind and $Q$ is the Legendre
  function of the second kind (see the NIST Digital Library of Mathematical
  Functions \cite{NIST:DLMF} \href{https://dlmf.nist.gov/14}{https://dlmf.nist.gov/14}). We here write $P(\nu,
  z)=P_{\nu}(z)=P_{\nu}^0(z)$).
  Most importantly, we are interested in the interval of $t$ where $t \xi
  \in(c^{-1},\xi)$ is very large and thus can use the asymptotics (\href{https://dlmf.nist.gov/14.8}{https://dlmf.nist.gov/14.8})
  \begin{align*}
    P(\nu, z) \sim c_{1}z^{\nu},\\
    Q(\nu, z) \sim c_{2}z^{-\nu-1},\\
  \end{align*}
  which for our purposes becomes $t^{\gamma_1-1}, t^{\gamma_2-1}$. Here,
  $c_{1,2}$ are explicit (if complicated) coefficients in terms of $\nu$.
  
  The corresponding values of $\psi_2$ can then be obtained from
  \begin{align*}
    \psi_2 = - \frac{\xi^{-2}+t^2}{c} \dt \psi_1,
  \end{align*}
  and the recurrence relation (see \href{https://dlmf.nist.gov/14.10}{https://dlmf.nist.gov/14.10})
  \begin{align*}
    (1-x^2)\frac{d}{dx}P(\nu, x)=  \nu P(\nu-1)(x)- \nu x P_{\nu}(x).
  \end{align*}
  Note that here $\nu \sim -\gamma_1,-\gamma_2$ was negative.
  Thus for our purposes these functions asymptotically behave as linear
  combinations of $(\xi t)^{\gamma_1}, (\xi t)^{\gamma_2}$.

  Given these homogeneous solutions, we may thus again make the ansatz that
  \begin{align*}
    u(t)= S(t)
    \begin{pmatrix}
      \alpha(t) \\ \beta(t)
    \end{pmatrix},
  \end{align*}
  where all entries of $S(t)$ are comparable to the power laws obtained above
  and $\det(S(t))=\det(S(t_0))$ is bounded away from zero.
  We remark that the conservation of the determinant follows from
  \begin{align*}
    \dt S + M S =0
  \end{align*}
  and noting that the trace of $M$ vanishes.
\end{proof}

Given the solution operator of problem \eqref{eq:9} we return to problem
\eqref{eq:twomode1} and introduce a particular solution to account for $u_3$.

\begin{lem}[Particular Solution]
  \label{lem:particular_solution}
  Consider the problem \eqref{eq:twomode1} with initial conditions
  \begin{align*}
    u_1(t_0)=u_2(t_0)=0.
  \end{align*}
  Then in the notation of Lemma \ref{lem:improvement} it holds that
  \begin{align*}
    |u_1(-\frac{d}{\xi})| \leq c|u_3(t_0)| C (\frac{d}{\xi})^{\gamma_2-1},\\
    |u_2(-\frac{d}{\xi})| \leq c|u_3(t_0)| C (\frac{d}{\xi})^{\gamma_2}.
  \end{align*}
  An analogous result holds on $I_3=(\frac{d}{\xi},t_1)$.
\end{lem}

\begin{proof}[Proof of Lemma \ref{lem:particular_solution}]

  Similarly to the proof of Lemma \ref{lem:improvement} we consider a variation
  of constants ansatz but this time using the solution operator $S(t)$ of the
  exact problem (that is, not approximating $b$).
  Here, we introduce
  \begin{align*}
    f=
    \begin{pmatrix}
      (b_1+b_2) u_3\\
      (b_1-b_2) u_3
    \end{pmatrix}
  \end{align*}
  in order to simplify notation. In particular recall that $u_3$ is conserved
  and hence $|f(t)|\leq 2 c |u_3(t_0)|$ for all times.
  
  We hence consider the ansatz
  \begin{align*}
    u(t)=S(t)
    \begin{pmatrix}
      \alpha(t) \\ \beta(t)
    \end{pmatrix},
  \end{align*}
  where $\alpha(t_0)=\beta(t_0)=0$ and $S(t)$ satisfies the power law bounds
  discussed in the proof of Lemma \ref{lem:improvement}. In particular,
  $|S^{-1}(t)|\leq C |t|^{\gamma_2-1}$ is integrable in time (by using the
  conserved determinant and Cramer's rule).

  Thus, it follows that
  \begin{align*}
    S(t) \dt     \begin{pmatrix}
      \alpha(t) \\ \beta(t)
    \end{pmatrix} = f, \\
    \Leftrightarrow \dt \begin{pmatrix}
      \alpha(t) \\ \beta(t)
    \end{pmatrix} = S(t)^{-1} f,
  \end{align*}
  and hence
  \begin{align*}
    (\alpha(t), \beta(t)) =\int_{t_0}^t \p_s (\alpha(s), \beta(s)) ds
  \end{align*}
  is bounded by $ C c |u_3(t_0)|$.
\end{proof}

Finally, we return to the proof of the a priori bound by a power law \eqref{eq:roughupper}.

\begin{lem}
  \label{lem:roughupperbound}
  Consider the problem \eqref{eq:9} on the interval $(t_0,-\delta)$ with $\xi^{-1}\ll \delta \ll 1$, then there exists a constant
  $C>0$ and $\sigma>0$ such that
  \begin{align*}
    |u_1(t)|\leq C |t|^{-\sigma} \|(u_1,u_2)(t_0)\|
  \end{align*}
  for all $t \in (t_0,-\delta)$. Here, $C$ and $\sigma$ may depend on $c$ but are
  uniform in $\xi$ and $\delta$.

  Analogously, for the interval $(\delta,t_1)$ it holds that
  \begin{align*}
    |u_1(t)|\leq C |t|^{-\alpha} \|(u_1,u_2)(t_1)\|.
  \end{align*}
\end{lem}

\begin{proof}[Proof of Lemma \ref{lem:roughupperbound}]
Our proof is based on concatenating estimates on a large number of
very small intervals, where we can use perturbative arguments.

Thus, consider an interval $I=(t_0,t_0+\nu) \subset (-1,\delta)$ with $t_0$
given and $\nu$ to be specified later.
Then on this interval it holds that
\begin{align*}
  |b(t)-b(t_0)| \leq 4 c \nu.
\end{align*}
We thus write our system as
\begin{align*}
  \dt
  \begin{pmatrix}
    u_1 \\ u_2
  \end{pmatrix}
+
  \begin{pmatrix}
    0 & a \\
    - b(t_0) & 0
  \end{pmatrix}
   \begin{pmatrix}
    u_1 \\ u_2
  \end{pmatrix}             
  = 
  \begin{pmatrix}
    0 & 0 \\
    b(t)-b(t_0) & 0
  \end{pmatrix}
   \begin{pmatrix}
    u_1 \\ u_2
  \end{pmatrix}. 
\end{align*}
Considering the right-hand-side as a perturbation, we make the ansatz
\begin{align}
  \label{eq:10}
     \begin{pmatrix}
    u_1 \\ u_2
  \end{pmatrix}=
 S(t)
  \begin{pmatrix}
    \alpha(t) \\\beta(t)
  \end{pmatrix},
\end{align}
where $S(t)$ is the solution operator to the problem with $b$ replaced by
$b(t_0)$. Then by Lemma \ref{lem:improvement} we know that component-wise
\begin{align*}
  S(t)\sim   \begin{pmatrix}
    -\frac{\gamma_1}{b(t_0)}|t|^{\gamma_1-1} & -\frac{\gamma_2}{b(t_0)} |t|^{\gamma_2-1} \\
    |t|^{\gamma_1} & |t|^{\gamma_2}
  \end{pmatrix}
\end{align*}
where $\gamma_{1,2}$ solve $\gamma_i(\gamma_i-1)=c b(t_0)$.

We further recall that $\det(S(t))$ is conserved due the vanishing trace in
equation \eqref{eq:9} and with the above comparison of the following size:
\begin{align*}
  \det(S)\sim \frac{\gamma_2-\gamma_1}{b(t_0)}.
\end{align*}
Hence, we may easily invert this matrix using Cramer's rule.

Plugging in this ansatz, we obtain that
\begin{align*}
  \dt  \begin{pmatrix}
    \alpha(t) \\\beta(t)
  \end{pmatrix}= S(t)^{-1}  \begin{pmatrix}
    0 & 0 \\
    b(t)-b(t_0) & 0
  \end{pmatrix} S(t) \begin{pmatrix}
    \alpha(t) \\\beta(t)
  \end{pmatrix}.
\end{align*}
We now consider the associated Duhamel integral and show that for $\nu$
sufficiently small it yields a small perturbation to the identity.
That is, we estimate
\begin{align*}
  & \quad \int_{t_0}^{t_0+\nu} \left\| S(t)^{-1}  \begin{pmatrix}
    0 & 0 \\
    b(t)-b(t_0) & 0
  \end{pmatrix} S(t)  \right\|_{op} dt \\
  &\leq  \frac{b_0}{\gamma_1-\gamma_2} 4c \nu \int_{t_0}^{t_0+\nu}\left\|
  \begin{pmatrix}
    |t|^{\gamma_2} & \frac{\gamma_2}{b(t_0)} |t|^{\gamma_2-1} \\
    -|t|^{\gamma_1}& -\frac{\gamma_1}{b(t_0)}|t|^{\gamma_1-1}
  \end{pmatrix}
                     \begin{pmatrix}
                       0 & 0 \\
                       1 & 0
                     \end{pmatrix}
           \begin{pmatrix}
    -\frac{\gamma_1}{b(t_0)}|t|^{\gamma_1-1} & -\frac{\gamma_2}{b(t_0)} |t|^{\gamma_2-1} \\
    |t|^{\gamma_1} & |t|^{\gamma_2}
  \end{pmatrix}                  
                     \right\|_{op}\\
  &= \frac{b_0}{\gamma_1-\gamma_2} 4c \nu \int_{t_0}^{t_0+\nu}\left\|
  \begin{pmatrix}
    0 & \frac{\gamma_2}{b(t_0)} |t|^{\gamma_2-1} \\
   0 & -\frac{\gamma_1}{b(t_0)}|t|^{\gamma_1-1}
  \end{pmatrix}
                     \begin{pmatrix}
                       0 & 0 \\
                       1 & 0
                     \end{pmatrix}
           \begin{pmatrix}
    -\frac{\gamma_1}{b(t_0)}|t|^{\gamma_1-1} & -\frac{\gamma_2}{b(t_0)} |t|^{\gamma_2-1} \\
    0 & 0
  \end{pmatrix}                  
                     \right\|_{op}.
\end{align*}
As $|t|<1$, we obtain an upper bound of the integrand by
\begin{align*}
  \frac{1}{b(t_0)^2} (\gamma_2^2 |t|^{2\gamma_2-2} + 2 \gamma_1 \gamma_2 |t|^{-1} + \gamma_1^2 |t|^{2\gamma_1-2})\\
  \leq  4 c^2 |t|^{-2} + 4|t|^{-1} + 4 c^{-2} |t|^{2\gamma_2-1}),
\end{align*}
where we used that $c\leq b(t_0)\leq 4 c$, $\gamma_1+\gamma_2=1$,
$\gamma_2\approx 4c^2$ (by Taylor's approximation).
Here, the last term is integrable in time, the middle term yields a logarithm
and the first term is estimated rather roughly.

We may hence control
\begin{align}
  \label{eq:11}
   \int_{t_0}^{t_0+\nu} \left\| S(t)^{-1}  \begin{pmatrix}
    0 & 0 \\
    b(t)-b(t_0) & 0
  \end{pmatrix} S(t)  \right\|_{op} \\
  \leq \frac{b_0}{\gamma_1-\gamma_2} 4c \nu \left(\frac{c^2}{t_0+\delta} + \ln(\frac{t_0+\delta}{t_0})+1\right).
\end{align}
Here we used that $t_0<t_0+\nu<0$ to bound the first integral by the larger
term. The case $0<t_0<t_0+\nu$ is switched accordingly.
We hence note that we cannot choose $t_0+\nu$ arbitrarily small, but rather
need to require that $\frac{t_0+\nu}{t_0}$ (respectively its reciprocal) is
not too large.
Choosing $\nu$ such that
\begin{align}
  \label{eq:12}
  t_0+\nu = \frac{1}{2} t_0,
\end{align}
we thus obtain that \eqref{eq:11} is bounded by $4c\ll 1$ and hence the map
\begin{align*}
  \begin{pmatrix}
    \alpha(t_0) \\\beta(t_0)
  \end{pmatrix} \mapsto
   \begin{pmatrix}
    \alpha(t_0+\nu) \\\beta(t_0+\nu)
  \end{pmatrix} 
\end{align*}
is comparable to the identity within a factor $2$, provided $\nu$ satisfies
\eqref{eq:12}.
It thus follows that $|u(t)|$ also only grows by a constant factor bounded by $100$ on the
interval $(t_0,t_0+\nu)=(t_0,t_0/2)$.
In order to establish the desired bound on all of $I_1=(t_0,\delta)$, we
partition our interval as
\begin{align*}
(t_0,t_0/2), (t_0/2,t_0/4), \dots, (2\delta,\delta).
\end{align*}
In order to reach a given $t \in I_1$, we then concatenate $|\ln(t/t_0)|/\ln(2)$
intervals and thus obtain that
\begin{align*}
  \|(u_1(t),u_2(t))\| \leq 100^{\ln(t/t_0)/\ln(2)}\|(u_1,u_2)(t_0)\| = |t/t_0|^{-\ln(100)/\ln(2)} \|(u_1,u_2)(t_0)\|.
\end{align*}
This concludes the proof of our rough upper bound with $\sigma=\ln(100)/\ln(2)$.
\end{proof}

\subsubsection{The interval $I_2$}
\label{sec:homogeneousintervals}

In order to study the evolution of \eqref{eq:homogeneous} on the middle interval
$I_2$ we use a different approach based on the convergent limit of iterated
Duhamel integration.
This method of proof also readily extends to the full (inhomogeneous) problem,
which is discussed in more detail in Proposition \ref{prop:middleinhom}.

\begin{lem}[Middle interval, homogeneous case]
\label{lem:middlehom}
  In the homogeneous case on the middle interval $I_2=(-\frac{d}{\xi},
  \frac{d}{\xi})$, $u_3$ is again conserved and
  \begin{align*}
    \begin{pmatrix}
      u_1(\frac{d}{\xi}) \\ u_2(\frac{d}{\xi})
    \end{pmatrix}
    \approx
    \begin{pmatrix}
      1 &  C c \xi \\
      \frac{2cd}{\xi} & 1
    \end{pmatrix}
     \begin{pmatrix}
      u_1(-\frac{d}{\xi}) \\ u_2(-\frac{d}{\xi}),
    \end{pmatrix}   
  \end{align*}
  where $C=\arctan(\xi t)|_{-\frac{d}{\xi}}^{\frac{d}{\xi}}= 2 \arctan(d)$.

  In particular, at time $t_2=+\frac{d}{\xi}$ it holds that
  \begin{align*}
    u_1(t_2)&\approx Cc \xi u_2(t_1) \approx 2C c^{1-\gamma_2} \xi^{1-\gamma_2} \frac{\gamma_1}{\gamma_2-\gamma_1} u_{2}(t_0), \\
    u_2(t_2)&\approx  u_2(t_1) \approx 2 c^{-\gamma_2} \xi^{-\gamma_2} \frac{\gamma_1}{\gamma_2-\gamma_1} u_2(t_0).
  \end{align*}
\end{lem}

\begin{proof}
  An explicit solution can be given in terms of hypergeometric functions,
  as we explore in the Appendix \ref{sec:special}.
  A more useful and shorter proof for our purposes considers an expansion in
  terms the Duhamel iteration.
  Here, it turns out that the first Duhamel
  iteration is dominant and all higher order Duhamel iterations can be estimated
  in a geometric series.
  Since the statement of the inhomogeneous case includes this result as a
  special case, in order to avoid duplication we refer to the proof of
  Proposition \ref{prop:middleinhom} for details.
  Concerning the asymptotics, we recall from Proposition \ref{prop:lefthom} that
  \begin{align*}
    u_1(t_1) &\approx c^{2-\gamma_2} \xi^{1-\gamma_2} u_2(t_0), \\
    u_2(t_1)& \approx c^{-\gamma_2} \xi^{-\gamma_2} u_2(t_0).
  \end{align*}
  Hence, it follows that
  \begin{align*}
       &\quad \begin{pmatrix}
      1 &  C c \xi \\
      \frac{2cd}{\xi} & 1
    \end{pmatrix}
     \begin{pmatrix}
      u_1(t_1) \\ u_2(t_1),
    \end{pmatrix}   \\
    &\approx u_2(t_0)
      \begin{pmatrix}
        1 & cC \xi \\
        \frac{2}{\xi} & 1
      \end{pmatrix}
                        \begin{pmatrix}
                          c^{2-\gamma_2} \xi^{1-\gamma_2} \\
                          2 c^{-\gamma_2}
                        \end{pmatrix} \\
       & \approx  u_2(t_0)
         \begin{pmatrix}
           2C c^{1-\gamma_2} \xi^{1-\gamma_2} \\
           2 c^{-\gamma_2} \xi^{-\gamma_2}
         \end{pmatrix}.
  \end{align*}
\end{proof}

\subsubsection{The interval $I_3$}
\label{sec:i3}

It remains to discuss the evolution on the interval $I_3=(\frac{d}{\xi},t_1)$.
\begin{lem}[Right interval, homogeneous case]
  \label{lem:righthom}
  Let $\xi \gg 1$ and $0<c <0.2$ be given. We consider the problem
  \eqref{eq:homogeneous} on the interval $I_3=(\frac{d}{\xi},t_1)$
  Then under the assumptions of Theorem \ref{thm:homogeneous_full} it holds that
  \begin{align*}
    u_1(t_1)&\approx c^{-1-\gamma_2} \xi^{-\gamma_2}u_{1}(\frac{d}{\xi}) - c^{1-\gamma_2} \xi^{1-\gamma_2} u_2(\frac{d}{\xi})\approx \xi^{\gamma} c^{1-2\gamma_2} \xi^\gamma u_2(t_0) \\
    u_2(t_1)& \approx 2c^{1-\gamma_2} \xi^{-\gamma_2}u_1(\frac{d}{\xi}) - 2c^{2-\gamma_2} \xi^{1-\gamma_2} u_2(\frac{d}{\xi}) \approx \xi^{\gamma} c^{2-2\gamma_2} u_2(t_0), \\
    u_3(t_1)&=u_3(t_0).
  \end{align*}
\end{lem}
  We note that $u_1(\frac{d}{\xi})$ is multiplied by a negative power
  of $\xi$ while $u_2(\frac{d}{\xi})$ is multiplied with a positive power of $\xi$.
\begin{proof}[Proof of Lemma \ref{lem:righthom}]
  This proof is largely analogous to the one of Proposition \ref{prop:lefthom}
  in the sense that the inverse solution operator $u(t_1)\mapsto
  u(\frac{d}{\xi})$ satisfies the same (asymptotic) estimates.

  As a first heuristic let us again consider the approximated two-mode model,
  where
  \begin{align*}
    \begin{pmatrix}
      u_1(\frac{d}{\xi}) \\
      u_2(\frac{d}{\xi})
    \end{pmatrix}
    \approx
    \begin{pmatrix}
      (\frac{d}{\xi})^{\gamma_1-1} & (\frac{d}{\xi})^{\gamma_2-1}\\
       \frac{\gamma_2}{c}(\frac{d}{\xi})^{\gamma_1} & \frac{\gamma_1}{c}(\frac{d}{\xi})^{\gamma_2} 
     \end{pmatrix} S(t_1)^{-1}
     \begin{pmatrix}
       u_1(t_1) \\ u_2(t_1)
     \end{pmatrix} \\
     \Leftrightarrow
          \begin{pmatrix}
       u_1(t_1) \\ u_2(t_1)
     \end{pmatrix} = S(t_1)
    \begin{pmatrix}
      (\frac{d}{\xi})^{\gamma_1-1} & (\frac{d}{\xi})^{\gamma_2-1}\\
       \frac{\gamma_2}{c}(\frac{d}{\xi})^{\gamma_1} & \frac{\gamma_1}{c}(\frac{d}{\xi})^{\gamma_2} 
     \end{pmatrix}^{-1}
     \begin{pmatrix}
      u_1(\frac{d}{\xi}) \\
      u_2(\frac{d}{\xi})
    \end{pmatrix}.
  \end{align*}
  Now note that
  \begin{align*}
    \begin{pmatrix}
       (\frac{d}{\xi})^{\gamma_1-1} & (\frac{d}{\xi})^{\gamma_2-1}\\
       \frac{\gamma_2}{c}(\frac{d}{\xi})^{\gamma_1} & \frac{\gamma_1}{c}(\frac{d}{\xi})^{\gamma_2} 
     \end{pmatrix}^{-1}
    = \frac{c}{\gamma_1-\gamma_2}
    \begin{pmatrix}
      \frac{\gamma_1}{c}(\frac{d}{\xi})^{\gamma_2}  & - (\frac{d}{\xi})^{\gamma_2-1} \\
      - \frac{\gamma_2}{c}(\frac{d}{\xi})^{\gamma_1} & (\frac{d}{\xi})^{\gamma_1-1} 
    \end{pmatrix}
  \end{align*}
  Thus, in this model we may compute  
  \begin{align*}
    \begin{pmatrix}
      u_1(t_1)\\ u_2(t_1)
    \end{pmatrix}
    &=
    \begin{pmatrix}
      t_1^{\gamma_1} & t_1^{\gamma_2} \\
      \frac{\gamma_2}{c}t_1^{\gamma_1-1} & \frac{\gamma_1}{c} t_1^{\gamma_2-1}
    \end{pmatrix}
                    \begin{pmatrix}
      \frac{\gamma_1}{c}(\frac{d}{\xi})^{\gamma_2}  & - (\frac{d}{\xi})^{\gamma_2-1} \\
      - \frac{\gamma_2}{c}(\frac{d}{\xi})^{\gamma_1} & (\frac{d}{\xi})^{\gamma_1-1} 
    \end{pmatrix}
                                                         \begin{pmatrix}
                                                           u_1(\frac{d}{\xi})\\ u_2(\frac{d}{\xi})
                                                         \end{pmatrix}\\
    &\approx
    \begin{pmatrix}
      \frac{\gamma_1}{c}(\frac{d}{\xi})^{\gamma_2} & -(\frac{d}{\xi})^{\gamma_2-1} \\
      2c (\frac{d}{\xi})^{\gamma_2} & -\frac{\gamma_2}{c} (\frac{d}{\xi})^{\gamma_2-1}
    \end{pmatrix}
    \begin{pmatrix}
      u_1(\frac{d}{\xi})\\ u_2(\frac{d}{\xi})
    \end{pmatrix}\\
    &=\begin{pmatrix}
      \frac{\gamma_1}{c}d^{\gamma_2} & -d^{\gamma_2-1} \\
      2c \ d^{\gamma_2} & -\frac{\gamma_2}{c} d^{\gamma_2-1}
    \end{pmatrix}
    \begin{pmatrix}
      \xi^{-\gamma_2} u_1(\frac{d}{\xi})\\ \xi^{1-\gamma_2}u_2(\frac{d}{\xi})
    \end{pmatrix}\\ 
    &\approx \begin{pmatrix}
      \frac{\gamma_1}{c}d^{\gamma_2} & -d^{\gamma_2-1} \\
      2c \ d^{\gamma_2} & -\frac{\gamma_2}{c} d^{\gamma_2-1}
    \end{pmatrix}
                          \begin{pmatrix}
                            2C c^{1-\gamma_2} \frac{\gamma_1}{\gamma_2-\gamma_1} \\
                            2 c^{-\gamma_2} \frac{\gamma_1}{\gamma_2-\gamma_1}
                          \end{pmatrix}
\xi^{\gamma}
   u_2(t_0),
  \end{align*}
  where we considered smaller powers of $\xi$ as errors and denoted
  $\gamma=1-2\gamma_2=2\sqrt{\frac{1}{4}-2c^2}=\sqrt{1-8c^2}$.
  Choosing $d=c^{-1}$ and recalling that $\gamma_2\approx 2c^2$ by Taylor's
  approximation, $\gamma_1 \approx \gamma \approx 1$ we may further approximate:
  \begin{align*}
    & \quad \begin{pmatrix}
      \frac{\gamma_1}{c}d^{\gamma_2} & -d^{\gamma_2-1} \\
      2c \ d^{\gamma_2} & -\frac{\gamma_2}{c} d^{\gamma_2-1}
    \end{pmatrix}
                          \begin{pmatrix}
                            2C c^{1-\gamma_2} \frac{\gamma_1}{\gamma_2-\gamma_1} \\
                            2c^{-\gamma_2} \frac{\gamma_1}{\gamma_2-\gamma_1}
                          \end{pmatrix} \\
    &\approx
    \begin{pmatrix}
      c^{-1-\gamma_2} & -c^{1-\gamma_2} \\
      2 c^{1-\gamma_2} & - 2c^{2-\gamma_2}
    \end{pmatrix}
                         \begin{pmatrix}
                           2C c^{1-\gamma_2} \\
                           2c^{-\gamma_2}
                         \end{pmatrix}\\
&=
    \begin{pmatrix}
      C c^{2\gamma_2} - 2 c^{1-2\gamma_2} \\
      2C c^{2-2\gamma_2} - 4 c^{2-2 \gamma_2}
    \end{pmatrix}
    \approx
    \begin{pmatrix}
      -2 c^{1-2 \gamma_2} \\
      \tilde{C} c^{2-2 \gamma_2}
    \end{pmatrix}.
  \end{align*}
  In particular, we stress that $u_1(t_1)\geq c^{-1} u_{2}(t_1)$ is much bigger
  provided $c$ is sufficiently small.

  It remains to discuss the extension to full problem.
  We note that under the transformation $(u_1(t),u_2(t),u_3(t))\mapsto
  (-u_1(-t),u_2(-t),u_3(-))$ we obtain an analogous problem as the one
  considered on $I_1$ in Section \ref{sec:upperandasymptotics} except that $t_0$
  is replaced by $-t_1$ and $b(t)$ is replaced by $b(-t)$.
  By the same arguments as in Section \ref{sec:upperandasymptotics} it hence
  follows that the solution operator $u(t):=S(t)u(t_1)$ is of the form 
  \begin{align*}
    S=
    \begin{pmatrix}
      S' &
      \begin{matrix}
        b \\ c
      \end{matrix}\\
      \begin{matrix}
        0 & 0
      \end{matrix}
& 1
    \end{pmatrix},
  \end{align*}
  where $S'$ is asymptotically well-approximated by our two-mode model heuristic.

  We may then explicitly compute the inverse of $S$ as 
  \begin{align*}
    S(t)^{-1}=
        \begin{pmatrix}
      S'^{-1} &
      \begin{matrix}
        \tilde{b} \\ \tilde{c}
      \end{matrix}\\
      \begin{matrix}
        0 & 0
      \end{matrix}
& 1
\end{pmatrix}, \\
    \begin{pmatrix}
      \tilde{b} \\ \tilde{c}
    \end{pmatrix}
    = - S'^{-1}
    \begin{pmatrix}
      b \\ c
    \end{pmatrix},
  \end{align*}
  where we can use Cramer's rule for inverting $S$.
  As we know that $S(\frac{d}{\xi})$ is well-approximated by the power laws the
  result for the full model hence follows by the exact same argument as we
  discussed for the approximate model above.
 \end{proof}

In the following Section \ref{sec:inhomogeneous} we use multiple bootstrap arguments to show that this
behavior persists in the full problem.
\subsection{The Full Model}
\label{sec:inhomogeneous}
\begin{thm}[Summary]
  \label{thm:summary}
  Let $c \in \R$ with $0<|c|<0.1$ and let $\xi \gg 1$.
  Let $u$ be the solution of the full model \eqref{eq:13}:
  \begin{align*}
    \dt u(k)+ a(k+1) u(k+1) - a(k-1)u(k-1)=0,
  \end{align*}
  on $(t_0,t_1)$.
    Suppose that at time $t_0$ it holds that $u(k_0)\geq 0.5 \max|u|=:0.5
    \theta$ and define $\gamma=\sqrt{1-8c^2}$.
    Then there exists a constant $C>0$ such that it holds that at time $t=1$
    \begin{align*}
      |u(k)| \leq C \xi^{\gamma} \theta
    \end{align*}
    for all $k$ and
    \begin{align*}
       |u(k_0- 1)| \geq \frac{C}{2} \xi^{\gamma} \theta.
    \end{align*}
    In particular, this theorem can be applied again.
\end{thm}

\subsubsection{The interval $I_1$}

\begin{prop}[Left interval, inhomogeneous case]
  \label{prop:leftinhom}
   Let $c \in \R$ with $0<|c|<0.1$ and let $\xi \gg 1$.
   Let $u$ be as in Theorem \ref{thm:summary}.
   Then at $-\frac{d}{\xi}=-\frac{1}{c\xi}$ it holds that:
   \begin{align}
     u(k_0) &\approx \theta \ c \ (\frac{d}{\xi})^{\gamma_2}, \\
     u(k_0\pm 1) &\approx \theta\  c \ (\frac{d}{\xi})^{\gamma_2-1}, \\
     u(k) &\leq C \theta,
   \end{align}
   where $C=(e^{c}-1)$
 \end{prop}

 \begin{proof}
   By linearity without loss of generality we may set $\theta=1$.

   Then at least for small positive time it holds that
   \begin{align}
     \label{eq:resonant}
      u(k_0) & \approx c u_2(-1) |t|^{\gamma_2}, \\
\label{eq:odd}
     (u(k_0+1)- u(k_0-1)) &\approx c u_2(-1) |t|^{\gamma_2-1}, \\
\label{eq:even}
     (\frac{u(k_0+1)+ u(k_0-1)}{2})_{t_0}^t &\leq e^{4.1 c(t-t_0)} -1 + (t^{\gamma_{2}}-t_0^{\gamma_{2}}), \\
     \label{eq:boundary}
     u(k) &\leq e^{2.1 c(t+1)} -1 + C (t^{\gamma_{2}}-t_0^{\gamma_{2}})\text{ else}, \\
     \label{eq:boundaryL2}
     \|u 1_{|k-k_0|\geq 2}\|_{l^2}(t) &\leq e^{4.1 c(t-t_0)}  \|u 1_{|k-k_0|\geq 2}\|_{l^2}(t_0).
   \end{align}
   In the following we argue by bootstrap that the maximal time satisfying these
   estimates is given by $T=\frac{1}{c\xi}$.

   Indeed, suppose there were $T<\frac{1}{c\xi}$ maximal with these properties.
   We then show that all conditions \eqref{eq:resonant} to \eqref{eq:boundary}
   do not achieve equality at time $T$ and that hence $T$ could be chosen larger
   by continuity, which contradicts the maximality.

   \underline{Ad \eqref{eq:even}}
   Similarly as in Lemma \ref{lem:reduction1} we may use the symmetry of the
   problem to compute 
   \begin{align*}
     &\dt (u(k_0+1)+ u(k_0-1)) = -a(k_0+2)u(k_0+2)+ a(k_0-2) u(k_0-2))\\
     \Rightarrow & |u(k_0+1)+ u(k_0-1)||_{t_0}^{t} \leq c \int_{t_0}^{T} (|u(k_0+2)|+ |u(k_0-2)|) \\
     & \leq  2c \int_{t_0}^{T} e^{4.1 c(t-t_0)} +(t^{\gamma_{2}}-t_0^{\gamma_{2}}).
   \end{align*}
   We may then roughly control $|t^{\gamma_{2}}-t_0^{\gamma_{2}}|\leq 1 \leq
   e^{4.1 c(t-t_0)}$ and thus control the integral by    
       \begin{align*}
    \frac{4}{4.1} (e^{2.1c(T-t_0)}-1)
     <(e^{4.1c(T-t_0)}-1).
   \end{align*}
   Thus, equality in \eqref{eq:even} is not attained at time $T$.

   \underline{Ad \eqref{eq:boundary}}
   First suppose in addition that $k\neq k_0+2, k_0-2$. Then it holds that
   \begin{align*}
     u(k)|_{t_0}^T= \int_{t_0}^T -a(k+1) u(k+1) + a(k-1) u(k-1) \leq 4c \int_{t_0}^T e^{4.1 c(t-t_0)} dt\\
     = \frac{4}{4.1} (e^{2c(T-t_0)}-1)<e^{2c(T-t_0)}-1,
   \end{align*}
   where we again estimated $|t^{\gamma_{2}}-t_0^{\gamma_{2}}|\leq 1 \leq
   e^{4.1 c(t-t_0)}$ and used that $|a(k)|\leq c$ unless $k=k_0$.
   
   If $k=k_0\pm 2$, $\int_{t_0}^T \p_t u(k)$ additionally involves 
   \begin{align*}
     \int_{t_0}^T a(k_0\pm 1) u(k_0\pm 1) dt \leq c t^{\gamma_{1,2}}|_{t_0}^T,
   \end{align*}
   which is controlled by $t^{\gamma_{1,2}}-t_0^{\gamma_{1,2}}$.

   \underline{Ad \eqref{eq:boundaryL2}:}
   Taking a time-derivative of the energy, that is the left-hand-side, we get that
   \begin{align*}
     \dt E(t)\leq 4c E(t) + c|u(k_0+2)||u(k_0+1)| \leq 4c E(t) + c (e^{4.1ct}+2)(t^{\gamma_{2}}-t_0^{\gamma_{2}}).
   \end{align*}
   The estimate hence follows by Gronwall's lemma or by multiplying with
   $e^{-4ct}$ and then integrating.

   The main part of the proof is thus given by the proof of the estimates for
   $u(k_0)$ and $u(k_0+1)-u(k_0-1)$.

   \underline{Ad \eqref{eq:resonant} and \eqref{eq:odd}}
   Following a similar argument as in the proof of Proposition \ref{prop:lefthom},
   we study the inhomogeneous problem for $(\frac{u(k_0+1)-u(k_0-1)}{2}, u(k_0))=: (u_1,u_2)$:
   \begin{align*}
     \dt
     \begin{pmatrix}
       u_1 \\ u_2
     \end{pmatrix}
     +
     \begin{pmatrix}
       0 & a \\
       b & 0
     \end{pmatrix}
     \begin{pmatrix}
       u_1 \\ u_2
     \end{pmatrix}
     &=- 
     \begin{pmatrix}
       a(k_0+2)u(k_0+2)/2- a(k_0-2)u(k_0-2)/2 \\ (a(k_0+1)-a(k_0-1)) (u(k_0+1)+u(k_0-1))/2 
     \end{pmatrix}\\
     &=: f.
   \end{align*}

   As in Lemma \ref{lem:improvement} we make a variation of constants ansatz,
   where for simplicity of notation write our calculations in terms of the power
   law solutions of the approximate model. By the estimates of Section
   \ref{sec:setup} the solution operator $S(t)$ of the full model is comparable
   and hence the exact same proof immediately extends.
   \begin{align*}
     \begin{pmatrix}
       u_1 \\ u_2
     \end{pmatrix}
     = S(t)
     \begin{pmatrix}
       \alpha(t) \\ \beta(t)
     \end{pmatrix}
     \approx \alpha(t)
     \begin{pmatrix}
       t^{\gamma_1-1}\\ t^{\gamma_1}\frac{\gamma_2}{c}
     \end{pmatrix}
     + \beta(t)      \begin{pmatrix}
       t^{\gamma_2-1}\\ t^{\gamma_2}\frac{\gamma_1}{c}
     \end{pmatrix}.
   \end{align*}
   Solving at time $t_0$, we obtain
   \begin{align*}
     \begin{pmatrix}
       \alpha(t_0)\\
       \beta(t_0)
     \end{pmatrix}
     &=
     S(t_0)^{-1}
                            \begin{pmatrix}
                              u_1(t_0) \\ u_2(t_0)
                            \end{pmatrix} \\
     &\approx \frac{1}{\sqrt{1-8c^2}}
     \begin{pmatrix}
       \gamma_1t_0^{\gamma_2-1} & -c t_0^{\gamma_2} \\
       - \gamma_2 t_0^{\gamma_1-1} & ct_0^{\gamma_1}
     \end{pmatrix}
                    \begin{pmatrix}
                              u_1(t_0) \\ u_2(t_0)
                            \end{pmatrix}.
   \end{align*}
   Here, in the proof of Proposition \ref{prop:lefthom} we concluded by noting that
   only $\beta(t_1)=\beta(t_1)\approx  c u_2(t_0)$ is relevant at $t_1$ (due to the
   smaller powers $\xi$ for $\alpha$).
In order to establish the desired bounds in this inhomogeneous case, we hence
need to show that $\alpha, \beta$ do not deviate from this too much.
We compute
   \begin{align*}
     \dt
     \begin{pmatrix}
       \alpha \\ \beta
     \end{pmatrix}
&= -c      \begin{pmatrix}
       t^{\gamma_1-1} & t^{\gamma_2-1}\\ t^{\gamma_1}\frac{\gamma_2}{c} &
t^{\gamma_2}\frac{\gamma_1}{c}
     \end{pmatrix}^{-1}   f\\
     &= -\frac{c^2}{\sqrt{1-8c^2}}
     \begin{pmatrix}
       t^{\gamma_2}\frac{\gamma_1}{c} & - t^{\gamma_2-1}\\
       - t^{\gamma_1}\frac{\gamma_2}{c} & t^{\gamma_1-1}
     \end{pmatrix} f.
   \end{align*}
   By estimates \eqref{eq:boundary} and \eqref{eq:even} we note that $|f|\leq c$
   and we further note that the integrals of $t^{\gamma_1}, t^{\gamma_1-1}$ are
   uniformly bounded, while $\int t^{\gamma_2-1}\leq \frac{1}{\gamma_2}\leq
   c^{-2}$, where we used that $\gamma_2 \approx 2 c^2$.
   Hence, in total it follows that $(\alpha, \beta)|_{t_0}^t \leq C (c,c^3)$ for
   all times and therefore $(\alpha(t),\beta(t))\approx (\alpha(t_0), \beta(t_0))$, which implies
   the result.
 \end{proof}

\subsubsection{The interval $I_2$}

 \begin{prop}[Middle interval, inhomogeneous case]
   \label{prop:middleinhom}
   Let $u$ be as in Theorem \ref{thm:summary}. Then at time $\frac{d}{\xi}$
   it holds that 
   \begin{align*}
     u(k_0\pm 1)|_{-d/\xi}^{d/\xi} &\approx C c\xi u(k_0,t_1) \approx \theta\ C c^{1-\gamma_2} \xi^{1-\gamma_2} , \\
     u(k_0)|_{-d/\xi}^{d/\xi} &\approx \frac{c}{\xi} u(k_0\pm 1, t_1) \approx \theta\ c^{-\gamma_2} \xi^{-\gamma_2},\\
     u(k)|_{-d/\xi}^{d/\xi}&\leq \frac{cd}{\xi} u(k_0\pm 1, t_1) \leq \theta\ c^{1-\gamma_2} \xi^{1-\gamma_2} ,
   \end{align*}
   where $C=2 \arctan(d)\approx \pi$.
 \end{prop}
We here use the Duhamel iteration, which we phrase as properties of the solution
map for single mode initial data. Arbitrary initial data can then be realized as
a linear combination.
We emphasize that on this short time interval the action $k_0 \mapsto k_0\pm 1$
is large (of size $C c \xi$), while all other actions are small perturbations of
the identity.
\begin{lem}
  \label{lem:res}
  Suppose that at time $-\frac{d}{\xi}$ it holds that $u(k)=\delta_{k k_0}$.
  Then at time $\frac{d}{\xi}$, $u$ satisfies
  \begin{enumerate}
  \item $|u(k_0)-1|\leq \frac{1}{1-c_1}\frac{c_2}{1-c_2}$, 
  \item $|u(k_0\pm 1)\mp 2c \xi \arctan(\xi t)|_{-d/\xi}^{d/\xi}|\leq \frac{2c}{1-4c^2d}
    2c \xi \arctan(\xi t)|_{t_0}^{t_1}$, 
  \item $|u(k)|\leq |\frac{c}{\xi}|^{|k-k_0|-1}\frac{1}{1-c}$ else,
  \end{enumerate}
  where $c_1=\frac{4cd}{\xi}=\frac{4}{\xi}$ and $c_2=8 c^2 d\arctan(\xi
  t)_{t_1}^{t_2}= 8c \arctan(d)\approx 4 \pi c$.
\end{lem}

\begin{proof}
  Ad (i): Let $\gamma=(k_0,\dots, k_0)$ be a path starting and ending in $k_0$,
  then we may roughly estimate
  \begin{align*}
    \iint_{t_0\leq \tau_1\leq \dots \leq t_1} \prod_{i} a(\gamma_i,\tau_i) \sgn(\gamma_{i+1}-\gamma_{i}) d\tau_i
  \end{align*}
  by
  \begin{align*}
    \prod \|a(\gamma_i,\tau_i)\|_{L^1([-d/\xi,d/\xi])} = (\frac{2cd}{\xi})^{j_1} (c \xi \arctan(\xi t)|_{-d/\xi}^{d/\xi})^{j_2},
  \end{align*}
  where $j_1$ corresponds to the number of non-resonance and $j_2$ to the number of
  resonances $\gamma_i=k_0$ (recall that the last entry of $\gamma$ does not appear in the
  integral).
  Then in order to start end in $k_0$ it needs to hold that $j_1\geq j_2$, since
  we have to come back to $k_0$ before leaving it again.
  Thus, the contribution of the path $\gamma$ is controlled by
  \begin{align*}
    (\frac{2cd}{\xi})^{j_1-j_2}  (2 c^2 d \arctan(\xi t)|_{-d/\xi}^{d/\xi})^{j_2}.
  \end{align*}
  Estimating the number of paths of given length $j_1+j_2$ by $2^{j_1+j_2}$ the
  sum over the integrals of all such paths can be controlled by
  \begin{align*}
    \frac{1}{1-\frac{4cd}{\xi}} \frac{8c^2d \arctan(\xi t)|_{-d/\xi}^{d/\xi}}{1-8c^2d \arctan(\xi t)|_{-d/\xi}^{d/\xi}},
  \end{align*}
  where we used that $j_2\geq 1$.

  Ad (ii): Let us first consider the special paths $\gamma=(k_0,k_0\pm 1)$,
  which yield an integral
  \begin{align*}
    \mp \int_{-d/\xi}^{d/\xi} a(k_0) d\tau= \mp c \xi \arctan(\xi t)|_{-d/\xi}^{d/\xi}\gg 1.
  \end{align*}
  For any other paths $\gamma$ starting in $k_0$ and ending in $k_0\pm 1$, we
  may again roughly bound the integral by
  \begin{align*}
    (\frac{2cd}{\xi})^{j_1} (c \xi \arctan(\xi t)|_{-d/\xi}^{d/\xi})^{j_2},
  \end{align*}
  where now $j_1\geq j_2-1\geq 0$ and we already treated $j_2=1,j_1=0$
  separately.
  We may thus express this bound as
  \begin{align*}
    (\frac{2cd}{\xi})^{\hat{j}_1}  (2 c^2 d \arctan(\xi t)|_{-d/\xi}^{d/\xi})^{\hat{j}_2}  (c \xi \arctan(\xi t)|_{-d/\xi}^{d/\xi}),
  \end{align*}
  where $(\hat{j}_1,\hat{j}_2)=(0,0)$ is excluded.
  Again estimating the number of such paths from above by
  $2^{\hat{j}_1+\hat{j}_2+1}$ and summing the geometric series, we obtain the
  desired result.

  Ad (iii): Let $k \not \in \{k_0-1,k_0,k_0+1\}$.
  Then given a path $\gamma=(k_0, \dots , k)$ there is a last time where
  $\gamma_i=k_0$, after which the remainder path is non-resonant at least $|k-k_0|$ times.
  Grouping all paths with the same remainder, we first estimate the contribution
  by the segments up to the last resonance as in (i) by
  \begin{align*}
    \frac{1}{1-\frac{4cd}{\xi}} \frac{8c^2d \arctan(\xi t)|_{-d/\xi}^{d/\xi}}{1-8c^2d \arctan(\xi t)|_{-d/\xi}^{d/\xi}}
  \end{align*}
  and then estimate the sum
  over all possible remainders by
  \begin{align*}
    \sum_{j\geq |k-k_0|} (\frac{4cd}{\xi})^{j}= \frac{1}{1-\frac{4cd}{\xi}} (\frac{4cd}{\xi})^{|k-k_0|},
  \end{align*}
  where we again introduced a factor $2^j$ to account for the number of all
  remainders of a given length.
\end{proof}

\begin{lem}
  \label{lem:nonres}
  Let $l \neq k_0$ and suppose that at time $-\frac{d}{\xi}$ it holds that
  $u(k)=\delta_{k l}$.
  Then $u$ satisfies
  \begin{align*}
    u(k)|_{-d/\xi}^{d/\xi}\leq \left(\frac{4cd}{\xi}\right)^{|k-l|} &+ \left(\frac{4cd}{\xi}\right)^{|k-k_0|+|l-k_0|-1}
    (8c^2 d\arctan(t \xi)|_{-d/\xi}^{d/\xi})\\
    &\quad \times \frac{1}{(1-\frac{4cd}{\xi})^2} \frac{1}{1-8c^2d\arctan(t\xi)|_{-d/\xi}^{d/\xi}}.
  \end{align*}
  Here, with slight abuse of notation $|a-b|$ denotes the minimal path length
  between $a$ and $b$, i.e. $|a-a|=2$.
\end{lem}

\begin{proof}
  Let $k\in \Z$ be given. Consider first the case of a purely non-resonant path
  $\gamma$, which can be estimated by
  \begin{align*}
    (\frac{2cd}{\xi})^{|\gamma|}.
  \end{align*}
  Any such path has length at least $|k-l|$ (with $|l-l|=2$).
  The first summand in the above estimate is hence obtained by a geometric
  series.

  Next consider a path with at least one resonance.
  We group all paths that share the segment from first to last resonance and
  observe that at least $|l-k_0|$ non-resonances are needed to reach $k_0$ for
  the first time and at least $|k-k_0|-1$ non-resonances to reach $k$ from
  $k_0$.
  We then may again control the sum over all path segments from $k_0$ to $k_0$
  as in the previous lemma and control the first and last segment by a geometric series.
\end{proof}

Using the result of the preceding lemmas, we can now prove Proposition \ref{prop:middleinhom}.

\begin{proof}[Proof of Proposition \ref{prop:middleinhom}]
  By linearity we may decompose our initial data $u(t_1)$ as
  \begin{align*}
    u(t_1)= u(t_1)|_{k=k_0} + u(t_1)|_{k\in \{k_0-1,k_0+1\}} + u^r(t_1).
  \end{align*}
  Then by Lemma \ref{lem:nonres} and the bound on $u^r(t_1)$ established in
  Proposition \ref{prop:leftinhom}, we may estimate its contribution to $u(t_2)$
  by
  \begin{align*}
    ((\frac{4}{\xi})^{|\cdot|} * u^r(t_1))(l) + Cc (\frac{4}{\xi})^{|l-k_0|}((\frac{4}{\xi})^{|\cdot-k_0|-1} * u^r(t_1))(l),
  \end{align*}
  where $C=\frac{16\arctan(d)}{(1-\frac{4}{\xi})^2(1-16c\arctan(d))}$.
  This contribution is thus very smooth and small and can be considered a
  perturbation.

  Concerning the contributions by $k_0$ and $k_0\pm 1$, we note that Lemmas
  \ref{lem:nonres} and \ref{lem:res} combined with the bounds on $u(t_1)$
  established in Proposition \ref{prop:leftinhom} control $u(k)$ for $k \in
  \{k_0-1,k_0,k_0+1\}$ in the desired way and show that
  \begin{align*}
    \begin{pmatrix}
      \frac{u(k_0+1)-u(k_0-1)}{2}\\
      u(k_0)
    \end{pmatrix}|_{t=t_2}
    \approx
    \begin{pmatrix}
      1 & 2Cc \xi \\
      \frac{1}{\xi} & 1
    \end{pmatrix}  \begin{pmatrix}
      \frac{u(k_0+1)-u(k_0-1)}{2}\\
      u(k_0)
    \end{pmatrix}|_{t=t_1}.
  \end{align*}
  We thus conclude as in the proof of Lemma \ref{lem:middlehom}.
\end{proof}

\subsubsection{The interval $I_3$}

\begin{prop}[Right interval, inhomogeneous case]
  Let $u$ be as in Theorem \ref{thm:summary}.
  Then it holds that at time $t_1$ all modes satisfy
  \begin{align*}
    u\leq C \xi^{\gamma}
  \end{align*}
  and $u(k_0\pm 1)$ is bounded below by $\frac{C}{2}\xi^{\gamma}$.
\end{prop}

\begin{proof}
  We again make a bootstrap ansatz similar to the one in Proposition
  \ref{prop:leftinhom}, but now take into account the different powers of $\xi$
  on different modes:
   \begin{align}
\label{eq:oddright}
     u(k_0+1)- u(k_0-1) &\approx 2 \theta c \xi^{\gamma} t^{1-\gamma_2} \\
     \label{eq:resonantright}
      u(k_0) &\approx \theta |\xi|^{\gamma} t^{\gamma_1}, \\
\label{eq:evenright}
     u(k_0+1)+ u(k_0-1)|_{d/\xi}^t &\leq \xi^{\gamma} (e^{4.1c(t-d/\xi)}-1 + t^{\gamma_{1}}-d/\xi^{\gamma_1}), \\
     \label{eq:boundaryright}
     u(k) &\leq \xi^{\gamma} (e^{4.1c(t-d/\xi)}-1+t^{\gamma_{1}}-d/\xi^{\gamma_1}) \text{ else}.
   \end{align}

   Ad \eqref{eq:boundaryright} and \eqref{eq:evenright}:
   Here, we again use that $u(k),\frac{u(k_0+1)+u(k_0-1)}{2} \leq \xi^{\gamma}$ at time
   $d/\xi$ and thus estimate
   \begin{align*}
     \int_{d/\xi}^t \dt u \leq 2c \xi^{\gamma} e^{4.1c(t+1)}<\xi^{\gamma} (e^{4.1c(t+1)}-1), 
   \end{align*}
   where we roughly estimate $t^{\gamma_{1}}-d/\xi^{\gamma_1}\leq 1 \leq
   e^{4.1c(t-d/\xi)}$ again.

   Ad \eqref{eq:oddright} and \eqref{eq:resonantright}:
   We make the same variation of constants ansatz as in the proof of Proposition
   \ref{prop:leftinhom} and again for simplicity of notation write the power law
   approximate solutions instead.
   Thus consider
   \begin{align*}
     \begin{pmatrix}
       u_1 \\ u_2
     \end{pmatrix}
     = S(t)
     \begin{pmatrix}
       \alpha(t) \\ \beta(t)
     \end{pmatrix}
     \approx \alpha(t)
     \begin{pmatrix}
       t^{\gamma_1-1}\\ t^{\gamma_1}\frac{\gamma_2}{c}
     \end{pmatrix}
     + \beta(t)      \begin{pmatrix}
       t^{\gamma_2-1}\\ t^{\gamma_2}\frac{\gamma_1}{c}
     \end{pmatrix}.
   \end{align*}
   Here, as in the proof of \ref{lem:improvement} in the following we present
   our argument in terms the approximate coefficients for simplicity of
   notation.
   
   We first compute $\alpha,\beta|_{d/\xi}$ by solving
   \begin{align*}
     \begin{pmatrix}
       \alpha \\ \beta
     \end{pmatrix}|_{t=\frac{d}{\xi}}
     =
     \begin{pmatrix}
       t^{\gamma_2}\frac{\gamma_1}{c} & - t^{\gamma_2-1}\\
       -t^{\gamma_1}\frac{\gamma_2}{c} & t^{\gamma_1-1}
     \end{pmatrix}|_{t=\frac{d}{\xi}}
                                         \begin{pmatrix}
                                           u_1 \\ u_2
                                         \end{pmatrix}|_{t=\frac{d}{\xi}}.                       
   \end{align*}
  
   Plugging in the relations between $\xi^{-\gamma_2} u_1(t_2)$ and
   $\xi^{1-\gamma_2} u_2(t_2)$, we see that this contribution satisfies the
   claimed estimates.
   
   It hence remains again to study the perturbations due to the inhomogeneity.
   As in the proof of Proposition \ref{prop:leftinhom} we compute
   \begin{align*}
     \p_t
     \begin{pmatrix}
       \alpha \\ \beta
     \end{pmatrix}
= -c
     \begin{pmatrix}
       t^{\gamma_1-1} & t^{\gamma_2-1} \\
        -t^{\gamma_1}\frac{\gamma_2}{c} & t^{\gamma_2}\frac{\gamma_1}{c}
      \end{pmatrix}^{-1}
                                          f.
   \end{align*}
   Plugging in our bounds \eqref{eq:evenright} and \eqref{eq:boundaryright} by
   $C \xi^{\gamma}$ into $f$ and integrating it
   follows that
   \begin{align*}
     \left|\begin{pmatrix}
       \alpha \\ \beta
     \end{pmatrix}|_{t=\frac{d}{\xi}}^{t} \right| \ll c \xi^{\gamma}.
   \end{align*}
   Thus, the value of $(\alpha,\beta)$ at time $\frac{d}{\xi}$ is dominant and
   satisfies the estimates.
\end{proof}

\subsection{Modified Scattering and Inviscid Damping}
\label{sec:mod}
\begin{thm}[Modified Scattering]
 Let $0<c<0.2$ be given. Then there exists $C_0=C_0(c)$ such that if $\omega_0 \in
 \mathcal{G}_{2,C}$ with $C>C_0$ then $\omega(t) \in \mathcal{G}_{2,C-C_0}$ globally
 in time and $u(t)$ converges in $\mathcal{G}_{2,C-C_0}$ as $t\rightarrow
 \infty$.\\
 
 On the other hand, for every $C<C_0/2$ and every $s \in \R$, there exists $\omega_0
 \in \mathcal{G}_{2,C}$ such that $\sup_{t \in [0,\infty)}
 \|\omega(t)\|_{H^{\sigma}}=\infty$ for any $\sigma\geq s$ but such that $\omega(t)$ converges
 in $H^{s-}$ as $t\rightarrow \infty$.
 \\
 
 Furthermore, for $s\geq 0$ the corresponding velocity field converges strongly
 in $L^2$ to a shear flow as $t\rightarrow \infty$.
Linear inviscid damping holds despite the divergence of $\omega(t)$ in higher regularity.  
\end{thm}

\begin{proof}
  We proceed similarly as in the case of Theorem \ref{thm:toyscattering}.
  Consider a frequency $\eta\gg 1$ and let $1\ll k \leq \sqrt{\eta}$ to be fixed
  later.
  Then by the local well-posedness established in Section \ref{sec:LWP} we may
  prescribe smooth initial data $\omega_0^{\eta,k}$ such that at the time $t=\frac{\eta}{k}$,
  $\omega$ is given by $e^{i\eta y + ik x}$.
  Then, we iteratively apply Theorem \ref{thm:summary} to
  obtain that after time $t=2\eta$, the mode $e^{i \eta y + i x}$ is the largest
  (within a factor) and of size
  \begin{align*}
    C^{k} \left( \frac{\eta^{k}}{(k!)^2} \right)^{\gamma},
  \end{align*}
  where we used that $\xi=\frac{\eta}{k^2}$ changes in each step.
  We may choose $k=k_\eta$ to maximize this product, which leads to factor
  $\exp(\tilde{C}\gamma \sqrt{\eta})=:g(\eta)$.
  Furthermore, by Theorem \ref{prop:largetime} after this time $t=2\eta$, the
  evolution is asymptotically stable and a small perturbation of the identity.

  Let now $\psi \in H^{s}(\R)$ be given and consider the initial datum:
  \begin{align*}
    \omega_0= \int_{\eta} \frac{1}{g(\eta)} \tilde{\psi}(\eta) \omega_0^{\eta,k_\eta}.
  \end{align*}
  Then by the definition of $g(\eta)$ and the properties of the evolution,
  $\omega(t)$ will asymptotically to leading order be given by
  \begin{align*}
    \omega_\infty = e^{ix}\int_{\eta} \tilde{\psi}(\eta) e^{i \eta y} d\eta= e^{ix} \psi(y).
  \end{align*}
  We can thus prescribe final data.
  We further observe that $\frac{1}{g(\eta)}\leq \exp(-C
  \sqrt{\eta})$ and thus $\omega_0 \in \mathcal{G}_{C,2}$.
  Our proof thus concludes by choosing $\psi \in H^{s}\setminus H^{\sigma}$
  appropriately. 
\end{proof}

\section{Discussion}
\label{sec:discussion}

In view of applications to the nonlinear dynamics we note that we have several competing
(de)stabilizing effects, whose interaction makes this a very challenging
problem:
\begin{itemize}
\item On the one hand the norm inflation mechanism of Section \ref{sec:echo} and
  similarly discussed in \cite{deng2018} shows that the vorticity may exhibit
  instability unless it is initially small in a sufficiently strong Gevrey
  class.
\item On the other hand resonant times are well-separated and for any given
  $\eta$ there are no resonances after time $\eta$. In particular, in the
  present problem fixing any
  finite radius $R$ as a frequency cut-off
  $\mathcal{F}^{-1}\chi_{B_R}\mathcal{F} \omega$ and its corresponding velocity
  field do converge irrespective of the regularity of the initial data.
\item Any instability will thus have to sustain a sequence of infinitely many
  separate echo chains for a sequence of times tending to infinity to ensure
  that the flow is not asymptotically stable after all (see Sections
  \ref{sec:LWP} and \ref{sec:mod}).
\item  While such a sequence of echo chains can be
  constructed in our model due to its decoupling structure (see Section
  \ref{sec:mod}), in the full nonlinear problem the conservation of enstrophy
  limits the possible relative growth. That is the conservation law imposes
  a hard ceiling for instability in that the $L^2$ energy remains bounded uniformly. Hence, it might be that the linear(!) instability mechanism of echoes is only applicable for finite times, after
  which the enstrophy limits further growth and the asymptotic stability of
  Section \ref{sec:LWP} takes over.

  Here the modified scattering results of Section \ref{sec:mod} and
  \cite{zillinger2018forced} provide a first indication that this may result in
  non-trivial but asymptotically stable behavior.
\end{itemize}
We further stress that, while stability of the linearized problem in
Sobolev \cite{Zill5}, \cite{Zill3}, \cite{Zhang2015inviscid} and Gevrey spaces \cite{jia2019linear} is
fundamental to attack the nonlinear problem, this article shows that it is further essential to
understand the linearization around non-shear low frequency perturbations, which
appear naturally in the nonlinear problem.

In the present work we have for simplicity of calculation and presentation
considered a single-mode perturbation $c\cos(x)$. We expect analogous results to also hold for more general finite sums
of small frequency perturbations, though involving quite involved calculations.
Our choice $\cos(x)$ is motivated by its simplicity and the fact that, as an
eigenfunction of the Laplacian, it is a stationary solution of the Euler
equations in Lagrangian coordinates with respect to Couette flow.

\appendix
\section{Special Functions and a Proof of Theorem \ref{thm:specialfunctions}}
\label{sec:special}

\begin{proof}[Proof of Theorem \ref{thm:specialfunctions}]
  Denoting $\xi=\frac{k^2}{\eta}$ for simplicity of notation, one may obtain the
  following explicit solution for boundary conditions $u(-1)=1, u'(-1)=0$ (e.g.
  using Mathematica)
  \tiny
  \begin{align}
    \label{eq:hypergeometric}    
\begin{split}
     \big(& -3 c^2 t \, _2F_1\left(\frac{3}{4}-\frac{1}{4} \sqrt{1-4 c^2},\frac{1}{4}
   \sqrt{1-4 c^2}+\frac{3}{4};\frac{3}{2};-\frac{1}{\xi ^2}\right) \,
   _2F_1\left(\frac{1}{4}-\frac{1}{4} \sqrt{1-4 c^2},\frac{1}{4} \sqrt{1-4
    c^2}+\frac{1}{4};\frac{3}{2};-\frac{t^2}{\xi ^2}\right) \\
   & -c^2 \,
   _2F_1\left(\frac{5}{4}-\frac{1}{4} \sqrt{1-4 c^2},\frac{1}{4} \sqrt{1-4
   c^2}+\frac{5}{4};\frac{5}{2};-\frac{1}{\xi ^2}\right) \, _2F_1\left(-\frac{1}{4}
   \sqrt{1-4 c^2}-\frac{1}{4},\frac{1}{4} \sqrt{1-4
     c^2}-\frac{1}{4};\frac{1}{2};-\frac{t^2}{\xi ^2}\right)\\
 &+3 \xi ^2 \,
   _2F_1\left(\frac{1}{4}-\frac{1}{4} \sqrt{1-4 c^2},\frac{1}{4} \sqrt{1-4
   c^2}+\frac{1}{4};\frac{3}{2};-\frac{1}{\xi ^2}\right) \, _2F_1\left(-\frac{1}{4}
   \sqrt{1-4 c^2}-\frac{1}{4},\frac{1}{4} \sqrt{1-4
     c^2}-\frac{1}{4};\frac{1}{2};-\frac{t^2}{\xi ^2}\right)\big) \\
 &/ \big(3 c^2 \,
   _2F_1\left(\frac{1}{4}-\frac{1}{4} \sqrt{1-4 c^2},\frac{1}{4} \sqrt{1-4
   c^2}+\frac{1}{4};\frac{3}{2};-\frac{1}{\xi ^2}\right) \,
   _2F_1\left(\frac{3}{4}-\frac{1}{4} \sqrt{1-4 c^2},\frac{1}{4} \sqrt{1-4
       c^2}+\frac{3}{4};\frac{3}{2};-\frac{1}{\xi ^2}\right)\\
   &-c^2 \,
   _2F_1\left(-\frac{1}{4} \sqrt{1-4 c^2}-\frac{1}{4},\frac{1}{4} \sqrt{1-4
   c^2}-\frac{1}{4};\frac{1}{2};-\frac{1}{\xi ^2}\right) \,
   _2F_1\left(\frac{5}{4}-\frac{1}{4} \sqrt{1-4 c^2},\frac{1}{4} \sqrt{1-4
       c^2}+\frac{5}{4};\frac{5}{2};-\frac{1}{\xi ^2}\right)\\
   &+3 \xi ^2 \,
   _2F_1\left(-\frac{1}{4} \sqrt{1-4 c^2}-\frac{1}{4},\frac{1}{4} \sqrt{1-4
   c^2}-\frac{1}{4};\frac{1}{2};-\frac{1}{\xi ^2}\right) \,
   _2F_1\left(\frac{1}{4}-\frac{1}{4} \sqrt{1-4 c^2},\frac{1}{4} \sqrt{1-4
   c^2}+\frac{1}{4};\frac{3}{2};-\frac{1}{\xi ^2}\right)\big)
    \end{split}
  \end{align}
  \normalsize
 Here $_2F_1$ denotes a hypergeometric function.
  We may then evaluate this formula at $t=1$ and use the series expansion of $_2
  F_1(a,b,c,x)$ at $x=\infty$:
  \begin{align}
    \label{eq:asymptotic}
    &x^{-a-b} \Bigg(x^b \left(\frac{(-1)^{-a} \Gamma (b-a) \Gamma (c)}{\Gamma (b) \Gamma
   (c-a)}+\frac{(-1)^{-a} a (a-c+1) \Gamma (b-a) \Gamma (c)}{(a-b+1) \Gamma (b)
    \Gamma (c-a) x}+O\left(\left(\frac{1}{x}\right)^2\right)\right)\\
    &+x^a
   \left(\frac{(-1)^{-b} \Gamma (a-b) \Gamma (c)}{\Gamma (a) \Gamma
   (c-b)}+\frac{(-1)^{-b} b (b-c+1) \Gamma (a-b) \Gamma (c)}{(-a+b+1) \Gamma (a)
   \Gamma (c-b) x}+O\left(\left(\frac{1}{x}\right)^2\right)\right)\Bigg).
  \end{align}
  This then for example yields that for $u(-1)=1,u'(-1)=0$,
   \tiny
  \begin{align*}
    u(1)&=  \Huge(-3 c^2 \, _2F_1\left(\frac{1}{4}-\frac{1}{4} \sqrt{1-4 c^2},\frac{1}{4} \sqrt{1-4 c^2}+\frac{1}{4};\frac{3}{2};-\frac{1}{r^2}\right) \, _2F_1\left(\frac{3}{4}-\frac{1}{4} \sqrt{1-4 c^2},\frac{1}{4} \sqrt{1-4 c^2}+\frac{3}{4};\frac{3}{2};-\frac{1}{r^2}\right) \\
      &-c^2  \, _2F_1\left(-\frac{1}{4} \sqrt{1-4 c^2}-\frac{1}{4},\frac{1}{4} \sqrt{1-4 c^2}-\frac{1}{4};\frac{1}{2};-\frac{1}{r^2}\right) \,
        _2F_1\left(\frac{5}{4}-\frac{1}{4} \sqrt{1-4 c^2},\frac{1}{4} \sqrt{1-4 c^2}+\frac{5}{4};\frac{5}{2};-\frac{1}{r^2}\right)\\
    &+3 r^2 \, _2F_1\left(-\frac{1}{4} \sqrt{1-4 c^2}-\frac{1}{4},\frac{1}{4} \sqrt{1-4 c^2}-\frac{1}{4};\frac{1}{2};-\frac{1}{r^2}\right)
          \, _2F_1\left(\frac{1}{4}-\frac{1}{4} \sqrt{1-4 c^2},\frac{1}{4} \sqrt{1-4 c^2}+\frac{1}{4};\frac{3}{2};-\frac{1}{r^2}\right)\Huge)\\
     &/ \Huge(3 c^2 \,
     _2F_1\left(\frac{1}{4}-\frac{1}{4} \sqrt{1-4 c^2},\frac{1}{4} \sqrt{1-4 c^2}+\frac{1}{4};\frac{3}{2};-\frac{1}{r^2}\right) \, _2F_1\left(\frac{3}{4}-\frac{1}{4} \sqrt{1-4 c^2},\frac{1}{4} \sqrt{1-4
       c^2}+\frac{3}{4};\frac{3}{2};-\frac{1}{r^2}\right)\\
     &-c^2 \, _2F_1\left(-\frac{1}{4} \sqrt{1-4 c^2}-\frac{1}{4},\frac{1}{4} \sqrt{1-4 c^2}-\frac{1}{4};\frac{1}{2};-\frac{1}{r^2}\right) \,
       _2F_1\left(\frac{5}{4}-\frac{1}{4} \sqrt{1-4 c^2},\frac{1}{4} \sqrt{1-4 c^2}+\frac{5}{4};\frac{5}{2};-\frac{1}{r^2}\right)\\
     &+3 r^2 \, _2F_1\left(-\frac{1}{4} \sqrt{1-4 c^2}-\frac{1}{4},\frac{1}{4}
     \sqrt{1-4 c^2}-\frac{1}{4};\frac{1}{2};-\frac{1}{r^2}\right) \, _2F_1\left(\frac{1}{4}-\frac{1}{4} \sqrt{1-4 c^2},\frac{1}{4} \sqrt{1-4 c^2}+\frac{1}{4};\frac{3}{2};-\frac{1}{r^2}\right)\Huge)
  \end{align*}
  \normalsize
can be approximated as 
\begin{align*}
  \frac{r^{-\sqrt{1-4 c^2}} r^2 \left(-\frac{3 \left(2^{\sqrt{1-4 c^2}-2} c^2 \left(\sqrt{1-4 c^2}+1\right) \Gamma \left(\frac{1}{2} \sqrt{1-4 c^2}\right)^2\right) }{\Gamma \left(\frac{1}{2} \left(\sqrt{1-4
   c^2}+3\right)\right)^2}\right)+o}{3 r^2+o},
\end{align*}
where we denoted $r=\xi^{-1}\ll 1$ and $o$ refers to terms decaying to higher
order in $r$. We note in particular that the powers $r^2$ cancel.
Approximating the value of the $\Gamma$ functions by their value in $c=0$, we thus obtain
\begin{align*}
  \xi^{\gamma} \pi c^2.
\end{align*}
Similar calculations for $u'$ and other initial data lead to the following
coefficient matrix:
\begin{align*}
  \begin{pmatrix}
    u(1)\\ u'(1)
  \end{pmatrix}\approx \xi^\gamma
  \begin{pmatrix}
    \pi c^2 & -\pi c^2 \\
    \pi c^2 & -\pi c^2
  \end{pmatrix}
        \begin{pmatrix}
    u(-1)\\ u'(-1)
  \end{pmatrix}         
\end{align*}
\end{proof}
While the above calculations and asymptotics are explicit, they are also rather opaque. 
    The splitting of the evolution into intervals $I_1,I_2,I_3$ studied in Lemmas \ref{lem:innerinterval},
    \ref{lem:schroedinger} and Proposition \ref{prop:exponentmechanism} instead
    provides a much clearer view of the underlying mechanism and yields the same
    leading asymptotics in terms of $\xi$.

\bibliographystyle{alpha}
\bibliography{citations2}
\end{document}